\documentclass[sn-mathphys,Numbered]{sn-jnl}

\usepackage{graphicx}%
\usepackage{multirow}%
\usepackage{amsmath,amssymb,amsfonts}%
\usepackage{amsthm}%
\usepackage{mathrsfs}%
\usepackage[title]{appendix}%
\usepackage{xcolor}%
\usepackage{textcomp}%
\usepackage{manyfoot}%
\usepackage{booktabs}%
\usepackage{algorithm}%
\usepackage{algorithmicx}%
\usepackage{algpseudocode}%
\usepackage{listings}%
\theoremstyle{thmstyleone}%
\newtheorem{theorem}{Theorem}%
\newtheorem{proposition}[theorem]{Proposition}%

\theoremstyle{thmstyletwo}%
\newtheorem{remark}{Remark}%

\theoremstyle{thmstylethree}%
\newtheorem*{claim}{Claim}
\newtheorem{lemma}{Lemma}[section]

\begin{document}

\title[$p-$Harmonic functions in the upper half-space]{$p-$Harmonic functions in the upper half-space}

\author[1]{\fnm{E.} \sur{Abreu}}\email{eabreu@ufmg.br}

\author[2]{\fnm{R.} \sur{Clemente}}\email{rodrigo.clemente@ufrpe.br}

\author*[3]{\fnm{J. M.} \sur{do \'O}}\email{jmbo@academico.ufpb.br}

\author[3]{\fnm{E.} \sur{Medeiros}}\email{everaldomedeiros1@gmail.com}

\affil[1]{\orgdiv{Departamento de Matem\'atica}, \orgname{Universidade Federal de Minas Gerais}, \orgaddress{\street{Av. Pres. Antônio Carlos}, \city{Belo Horizonte}, \postcode{30161-970}, \state{Minas Gerais}, \country{Brazil}}}

\affil[2]{\orgdiv{Departamento de Matem\'atica}, \orgname{Universidade Federal Rural de Pernambuco}, \orgaddress{\street{Rua Dom Manuel de Medeiros}, \city{Recife}, \postcode{58051-900}, \state{Pernambuco}, \country{Brazil}}}

\affil*[3]{\orgdiv{Department of Mathematics}, \orgname{Federal University of Para\'{\i}ba}, \orgaddress{\street{Campus I}, \city{João Pessoa}, \postcode{58051-900}, \state{Para\'iba}, \country{Brazil}}}

\abstract{This paper investigates the existence, nonexistence, and qualitative properties of p-harmonic functions in the upper half-space $\mathbb{R}^N_+ \, (N \geq 3)$ satisfying nonlinear boundary conditions for $1<p<N$. Moreover, the symmetry of positive solutions is shown by using the method of moving planes.}


\maketitle

\section{Introduction}
This work treats some aspects of the $p-$harmonic equations in the upper half-space under nonlinear boundary conditions.  The discussion focuses on the existence, nonexistence and qualitative properties of solutions for the following model of quasilinear elliptic problems with nonlinear boundary conditions,
\begin{equation}\label{P1}
	\left\{
	\begin{alignedat}{2}
& \Delta_p u \equiv \mathrm{div}(|\nabla u|^{p-2} \nabla u)		=0\quad&\text{ in }&\quad\mathbb{R}_{+}^{N},\\
	&	\vert \nabla u \vert^{p-2}\frac{\partial u}{\partial \nu}+\lambda |u|^{p-2}u= |u|^{q-2}u\quad&\text{ on }&\quad\mathbb{R}^{N-1},
	\end{alignedat}
	\right.
	\tag{$\mathcal{P}_\lambda$}
\end{equation}
where $\mathbb{R}^N_+:=\{x=(x',x_N)\in\mathbb{R}^N :  x' \in \mathbb{R}^{N-1},   x_N>0\}$ standards for the upper half-space, $\lambda$ is a real parameter, $\nu$ is the unit outer normal to the boundary $\partial\mathbb{R}^N_+:=\mathbb{R}^{N-1}, \; 1<p<N\quad\mbox{and}\quad p\leq q<\infty.$ The nonlinear differential operator $\Delta_p u \equiv \mathrm{div}(|\nabla u|^{p-2} \nabla u)$ is known as 
the $p-$Laplacian.

Nonlinear elliptic equations have been the object of intense study motivated by many problems in several mathematical physics and geometry branches. 
In particular, many works have been concerned with the existence and qualitative properties of solutions for problems involving  elliptic equations with nonlinear boundary conditions in the upper half-space. For instance,  the works  \cite{liu,terra,chipot,Hu,EMEDOOEVE2010} investigated nonlinear elliptic problems involving nonlinear boundary conditions.

Problem \eqref{P1} is related to the Euler-Lagrange equation associated with the Sobolev embedding 
 \begin{align}\label{Nossa inequality}
 	{\displaystyle\left(\int_{\mathbb{R}^{N-1}}|u|^q \mathrm{d}  x'\right)^{p/q}}\leq S^{-1}_\lambda(p,q)\displaystyle\left(\int_{\mathbb{R^N_+}}|\nabla u|^p \mathrm{d}  x+\lambda\int_{\mathbb{R}^{N-1}}|u|^p \mathrm{d}  x'\right),\quad u\in  C_0^\infty(\mathbb{R}^N),
 \end{align}
 where $S_\lambda(p,q)$ is the best constant for \eqref{Nossa inequality}. 
 J. Escobar in the remarkable paper \cite{Escobar}, by exploiting the conformal  invariance of \eqref{Nossa inequality}  when $p=2, \; \lambda=0$ and $N\geq3$,   characterize the minimizers of $S_0(2,2_*)$. He also conjectured if similar result holds for minimizers of $S_0(p,p_*)$, which was proved by B. Nazaret \cite{Nazaret} for the case $N\geq 3$ and $1<p<N$ based on the mass transportation approach.
 
 We also mention that a significant number of authors studied inequality \eqref{Nossa inequality} in bounded domains (or compact manifolds) see \cite{yanyan,EMEDOOEVE2010} and their references. 
 
Motivated by the above discussion, the primary purpose of this paper is threefold. First, we use min-max arguments to investigate the existence of ground state solutions to \eqref{P1}, that is, solutions that admit the smallest energy among all nontrivial weak solutions. Second, we will deal with the asymptotic decay results and symmetric properties for positive solutions to \eqref{P1}. Moreover,  we prove a nonexistence result by applying a Poho\v{z}aev type identity.

\subsection{Statement of the main results}    
	Assume that $1<p<N$. The approach to study the existence of solutions for \eqref{P1}  is based on variational methods, and for that,  we consider the subspace of $\mathcal{D}^{1,p}(\mathbb{R}_{+}^{N})$ defined by
	\begin{equation*}
	    \begin{aligned}
	   & E:=\left\{ u\in \mathcal{D}^{1,p}(\mathbb{R}_{+}^{N})\text{ : }u\vert_{\mathbb{R}^{N-1}}\in L^p(\mathbb{R}^{N-1})\right\},\\
	   & \Vert u \Vert_{E}=
	   \left( \|\nabla u\|^p_{L^p(\mathbb{R}_{+}^{N})} + \|u\|^p_{L^p(\mathbb{R}^{N-1})}\right)^{1/p},
	     \end{aligned}
	\end{equation*}
	where $ u\vert_{\mathbb{R}^{N-1}}$ is understood in the trace sense. The space $\mathcal{D}^{1,p}(\mathbb{R}_{+}^{N})$ is the completion of the restriction on $\mathbb{R}_{+}^{N}$ of functions $C^{\infty}_{0}(\mathbb{R}^N)$ with respect to the norm $\|\nabla u\|_{L^p(\mathbb{R}_{+}^{N})}$.	
	
By a weak solution of \eqref{P1},  we understand  a function $u\in E\setminus\{0\}$ such that for all $\varphi\in C_0^\infty(\mathbb{R}^N),$
	\begin{equation}\label{Def1}
		\int_{\mathbb{R}_{+}^{N}}\vert \nabla u \vert^{p-2}\nabla u\nabla \varphi\,\mathrm{d}x+\lambda\int_{\mathbb{R}^{N-1}}\vert u \vert^{p-2}u\varphi\,\mathrm{d}x'=\int_{\mathbb{R}^{N-1}}|u|^{q-2}u\varphi\,\mathrm{d}x'.
	\end{equation}
	
	\begin{theorem}\label{T-Existence}
		Assume $p< q< p_*:=p(N-1)/(N-p)$. Then, Problem \eqref{P1} has a positive ground state solution for all $\lambda>0$. 
	\end{theorem}

\medskip

\begin{remark}
 \begin{flushleft}
$\mathrm{(a)}$  One can see that   $S_\lambda(p,q)>0$ and it is achieved for all $\lambda>0$, as a consequence of Theorem~\ref{T-Existence}. \\
 $\mathrm{(b)}$ If $p<q<p_*$ it holds $S_\lambda(p,q)\rightarrow 0$ as $\lambda\rightarrow 0^+$, see Lemma \ref{111} below.\\  
$\mathrm{(c)}$ If $q=p_*$, we have $S_\lambda(p,p_*)\rightarrow S_0(p,p_*)$ as $\lambda\rightarrow 0^+$, where $S_0(p,p_*)$ is the best constant of the Sobolev trace embedding $\mathcal{D}^{1,p}(\mathbb{R}^N_+)\hookrightarrow L^{p^*}(\mathbb{R}^{N-1}).$\\
$\mathrm{(d)}$ Let $m,p\in (1,N)$, $q\in (p,p_*)$ and $\lambda \geq 0$, then $S_\lambda (p,q)\rightarrow S_\lambda (m,q)$ as $p\rightarrow m$. \\
$\mathrm{(e)}$ If $\lambda=0$, we have $S_0(p,q)>0$ if and only if $q= p_*,$  by using the trace embedding and a scaling procedure. For details, see \cite{Nazaret}.
\end{flushleft}
\end{remark}

\medskip

	Next, we investigate regularity and asymptotic behavior for solutions of Problem \eqref{P1}. By using a Moser iteration procedure and a Harnack type inequality  (see Lemma \ref{12}), we show that weak solutions of Problem \eqref{P1} decay to zero at infinity. 
	To be precise, we state the following result: 

\medskip
 
	\begin{theorem}\label{T-Regularity-Decay}
		Let $u_\lambda$ be a weak solution of Problem \eqref{P1} with $p< q< p_*$. Then, in the trace sense $u_\lambda\vert_{\mathbb{R}^{N-1}}\in L^\infty(\mathbb{R}^{N-1})$ and  $u_\lambda\in L^\infty(\mathbb{R}_{+}^{N})$. Consequently, weak solutions of \eqref{P1} are of class $C^{1,\alpha}_{\mathrm{loc}}(\overline{\mathbb{R}^N_+})$ for some $0<\alpha<1$.
		Furthermore, $u_\lambda$ has the following decay rate at infinity,
		\[
		u_\lambda(x)=O\left( \vert x\vert^{\frac{p-N}{p-1}} \right)\quad\text{ as }\quad \vert x \vert\rightarrow +\infty.
		\]
	\end{theorem}

\medskip

	On the nonexistence of solutions for Problem \eqref{P1}, we mention the result proved by B.~Hu \cite{Hu} for the particular case that $\lambda=0$  and $p=2$. Here, we complete his analysis by using a Poho\v{z}aev type identity. Indeed, we show the nonexistence results of weak solutions stated as follows.

 \medskip
	
	\begin{theorem}\label{T0} Let $u_\lambda\in E\cap C^1(\overline{\mathbb{R}^{N}_+})$ be a weak solution of Problem \eqref{P1}. Then $u_\lambda\equiv0$ if one of the conditions hold
 \begin{flushleft}
{$\mathbf {i)}$}   If $\lambda=0$ and $q\in[p,p_*)\cup(p_*,+\infty)$,\\
{$\mathbf {ii)}$} If $\lambda>0$ and $q\in[p_*,+\infty)$,\\
{$\mathbf {iii)}$} If $\lambda<0$ and $q\in[p,p_*]$.
  \end{flushleft}
Moreover, when the solution $u_\lambda$ is nontrivial, it holds
		$\lambda\leq\|u_\lambda\|^{q-p}_{L^\infty(\mathbb{R}^{N-1})}$.
	\end{theorem}

\medskip
 
The celebrated moving plane method (MPM) due to A. Alexandrov \cite{alex,alex2}  is an essential tool to study the symmetry of solutions for nonlinear elliptic problems.  See the classical works of  J. Serrin \cite{serrin}, and B. Gidas, M. Ni and L. Nirenberg \cite{gnn}.  The ideas used to study the Laplacian case can not be applied directly to analyze symmetry for solutions of the $ p-$Laplacian equations because this elliptic operator is singular or degenerate if $1 < p < 2$ or $p > 2$, respectively, on the critical set $\{ \nabla u=0 \}$, and comparison principles are not equivalent to maximum principles, as for the semilinear case.  In \cite{damas,damas2}, the authors have used some adaptation of the moving plane technique to study the case of the p-Laplace operator in bounded domains. For the case when the domain is the half-space, symmetric properties of positive solutions to $-\Delta_pu = f (u)$ with zero Dirichlet assumption was proved in \cite{farina,farina2}.  We also mention an important result of symmetry related to the Laplacian equation with nonlinear boundary conditions is due to Y.~Li and M.~Zhu \cite{yanyan}. They used the method of moving spheres, a variant of the method of moving planes, to prove that nontrivial nonnegative solutions of boundary value problems of the form $-\Delta u=f(u)$ in $\mathbb{R}^N_+$ satisfying $\partial u/\partial x_N=g(u)$ on $\partial{\mathbb{R}}^N_+$ must take a specific form. In particular, for $N\geq3$, they proved that the solution must take the form $u(x^\prime,x_N)=(\epsilon/[\epsilon^2+|(x^\prime,x_N)-(x^\prime_0,x_{N_0})|^2])^{(N-2)/2}$ with $\epsilon>0$ if $f(u)=$ $N(N-2)u^{(N+2)/(N-2)}$ and $g(u)=cu^{N/(N-2)}$. For more results about existence, symmetry, and qualitative properties for semilinear elliptic problems with nonlinear boundary conditions, we refer the reader to \cite{harada,terra,chipot}.
We quote here that for $p=2$, E.~Abreu et al. \cite{EMEDOOEVE2010} proved that Problem \eqref{P1} has a ground state solution $w$ which is radially symmetric with respect to $x^\prime\in \mathbb{R}^{N-1}$, and decay to zero like $x_N^{2-N}$ at infinity.

Next, we state our symmetric result for positive solutions of \eqref{P1}, which  is the first one when the domain is the half-space and involves nonlinear boundary conditions to the best of our knowledge.
\medskip

	\begin{theorem}\label{simetria}
		Let $u$ be a positive weak solution of Problem \eqref{P1} with $p< q< p_*$. Then $u$ is radially symmetric with respect to $(N-1)$ first variables, that is 
		\[u(x^\prime,x_N)=u(r,x_N), \ \forall\ (x^\prime,x_N)\in\mathbb{R}^{N-1}\times\mathbb{R}_+, \ \mbox{with}\ \ |x^\prime|=r.\]
	\end{theorem}

	\subsection{Outline}
	The paper is organized as follows. The following section brings preliminary results and a variational framework related to \eqref{P1}. In Section \ref{existence}, we use the min-max argument to prove Theorem \ref{T-Existence}. Section \ref{reg} establishes the $L^\infty$-estimate based on Moser’s iteration, a Harnack type inequality, and Theorem \ref{T-Regularity-Decay}. Section \ref{poho} is devoted to proving a Poho\v{z}aev type identity and establishing a nonexistence result for positive regular decaying solutions \eqref{P1}. The last section is devoted to proving symmetry for solutions to \eqref{P1} using the celebrated method of moving planes.
	\section{Preliminary results} 
	First, we give a brief presentation of the key ingredients of the analysis and recall the main properties of some Sobolev spaces. 
	As was already mentioned, the first goal of this work is to show the existence of positive ground state solutions to \eqref{P1}, by using variational methods based on variants of the minimax theorem. For that, it is natural to consider the  associated functional  $I:E \rightarrow \mathbb{R}$ given by 
		\[
	I(u)=\frac{1}{p}\int_{\mathbb{R}_{+}^{N}}\vert \nabla u \vert^p\,\mathrm{d}x+\frac{\lambda}{p}\int_{\mathbb{R}^{N-1}}\vert u\vert^p\,\mathrm{d}x'-\frac{1}{q}\int_{\mathbb{R}^{N-1}}(u^+)^q\,\mathrm{d}x',
	\]
	where $u^+(x)=\max\left\{u(x),0\right\}$. The functional $I$ is well defined in view of the embedding
	\begin{equation}\label{emb1}
	   E\hookrightarrow L^{q}(\mathbb{R}^{N-1})\quad\mbox{for any}\quad p\leq q\leq p_*. 
	\end{equation}
	Using Gagliardo-Nirenberg-Sobolev inequality, we have the continuous embedding 
	\begin{equation}\label{Escobar}
		E\hookrightarrow\mathcal{D}^{1,p}(\mathbb{R}_{+}^{N}),
	\end{equation}
	which, together with the trace embedding
	$$ 
	\mathcal{D}^{1,p}(\mathbb{R}_{+}^{N})\hookrightarrow L^{p_*}(\mathbb{R}^{N-1}),\quad p_*:=\frac{(N-1)p}{N-p},
	$$
	gives $	E\hookrightarrow L^{p_*}(\mathbb{R}^{N-1}).$ Since by definition $E\hookrightarrow L^{p}(\mathbb{R}^{N-1}),$ one can apply interpolation inequality in Lebesgue spaces to get  \eqref{emb1}.
	
	Using standard arguments, we can see that $I\in C^1(E,\mathbb{R})$ and its critical points correspond to weak solutions of \eqref{P1}. Next, we prove a density result involving the space $E$.
	
	\begin{lemma}\label{este2} $W^{1,p}(\mathbb{R}^N_+)$ is dense in $E$.
	\end{lemma}
	\begin{proof}
		For $R>0$ let us consider $\varphi_R \in C^{\infty}_{0}(\mathbb{R}^N, \mathbb{R})$ such that $0\leq\varphi_R\leq1, \; |\nabla\varphi_R|\leq1/R$ and
		$$
		\varphi_R(x',x_N)=
		\left\{
		\begin{aligned} &1,\quad\mbox{if}\quad  |(x',x_N)|\leq R,\\
		&0\quad\mbox{if}\quad|(x',x_N)|\geq 2R,
		\end{aligned}
		\right.
		$$
		Given
		$u\in E$ using \eqref{Escobar},  we have  $u_R=u\varphi_R\in W^{1,p}(\mathbb{R}^N_+)$ and
		$$
		\begin{aligned}
		\|u-u_R\|^p_{E}&=\int_{\mathbb{R}^N_+}|\nabla u-\nabla
		u_R|^p\,\mathrm{d}x+\int_{\mathbb{R}^{N-1}}|u-u_R|^p\,\mathrm{d}x^\prime\\
		&=\int_{(B_R^+)^C}|\nabla u-\nabla u_R|^p\,\mathrm{d}x+\int_{{|x'|>R}}|u-u_R|^p\,\mathrm{d}x^\prime\\
		&\leq \int_{(B_R^+)^C}|\nabla u|^p\,\mathrm{d}x+\int_{(B_R^+)^C}|\nabla
		u_R|^p\,\mathrm{d}x+C\int_{{|x'|>R}}|u|^p\,\mathrm{d}x^\prime\\
		&=o_R(1)+\int_{(B_R^+)^C}|\nabla
		u_R|^p\,\mathrm{d}x+o_R(1),
		\end{aligned}
		$$
		where $o_R(1)\rightarrow0$ as $R\rightarrow\infty, $ \begin{equation*}
		 B_R^+:=\{x\in \mathbb{R}^N_+: |(x^\prime,x_N)|\leq R\} \text{ and } (B_R^+)^C:=\{x \in  \mathbb{R}^N_+: |(x^\prime,x_N)| > R\} \}.
		\end{equation*}
	    We claim that
		\[
		\int_{(B_R^+)^C}|\nabla u_R|^p\,\mathrm{d}x=o_R(1) \text{ as }R\rightarrow \infty.
		\]
		Indeed,
		$$
		\begin{aligned}
		\int_{(B_R^+)^C}|\nabla u_R|^p\,\mathrm{d}x&\leq
		C\left(\int_{(B_R^+)^C}|\nabla
		u|^p\varphi_R^p\,\mathrm{d}x+\int_{(B_R^+)^C}|u|^p|\nabla\varphi_R|^p\,\mathrm{d}x\right)\\
		&\leq C\left(\int_{(B_R^+)^C}|\nabla
		u|^p\,\mathrm{d}x+\frac{1}{R^p}\int_{A_{R,R+1}^+}|u|^p\,\mathrm{d}x\right)\\
		&=o_R(1)+\frac{C}{R^p}\int_{A_{R,2R}^+}|u|^p\,\mathrm{d}x,
		\end{aligned}
		$$
		where, $A_{R,2R}^+:=\{(x',x_N)\in\mathbb{R}^N_+; R^2\leq|x'|^2+x_N^2\leq 4R^2 \}$.	
		Using embedding \eqref{Escobar} we see that
		$$
		\begin{aligned}
		\frac{1}{R^p}\int_{A_{R,2R}^+}|u|^p\,
		\mathrm{d}x&\leq\frac{1}{R^p}
		\left(\int_{A_{R,2R}^+}|u|^{p^*}\,
		\mathrm{d}x\right)^{(N-p)/N}
		\left(\int_{A_{R,2R}^+}\,\mathrm{d}x \right)^{p/N}\\
		&\leq\frac{1}{R^p}\|u\|_{E}^p\left(\int_{A_{R,2R}^+}\,\mathrm{d}x \right)^{p/N},
		\end{aligned}
		$$
		where $1/p^*=1/p+1/N$. Since
		$$
		\int_{A_{R,2R}^+}\mathrm{d}x=\frac{C_n}{2}[(2R)^N-R^N]=\frac{C_N}{2}N\theta^{N-1}\leq \widetilde{C}_NR^{N-1},
		$$
		we obtain
		$$
		\begin{aligned}
		\frac{1}{R^p}\int_{A_{R,2R}^+}|u|^p\mathrm{d}x&\leq \frac{C}{R^p}\left(R^{N-1}\right)^{p/N}=\frac{C}{R^{p/N}}\rightarrow 0\quad \mbox{as}\quad R\rightarrow\infty,
		\end{aligned}
		$$
		and this completes the proof of Lemma~\ref{este2}.
	\end{proof}
	
	\begin{lemma}\label{0006} The set of restrictions to $\mathbb{R}_+^N$ of functions in $C_0^\infty(\mathbb{R}^N)$ is dense in $E$.
	\end{lemma}
	\begin{proof} Fixed $u\in {E}$ and $\varepsilon>0$, by Lemma~\ref{este2} there exists
		$u_1\in W^{1,p}(\mathbb{R}^N_+)$ such that
		\begin{equation}\label{esta1}
			\|u-u_1\|_{E}\leq \varepsilon.
		\end{equation}
		Using that the set of restrictions to $\mathbb{R}_+^N$ of functions in $C_0^\infty(\mathbb{R}^N)$  is dense
		in $W^{1,p}(\mathbb{R}^N_+)$ (see Theorem 3.22 \cite{adams}), there exists $v\in
		C_0^\infty(\mathbb{R}^N)$ such that
		$
		\|u_1-v\|_{W^{1,p}(\mathbb{R}^N_+)}\leq \varepsilon.
		$
		By the trace embedding theorem, we have
		$W^{1,p}(\mathbb{R}^N_+)\hookrightarrow E$. Consequently, we
		have
		\[
		\|u_1-v\|_{E}\leq C\|u_1-v\|_{W^{1,p}(\mathbb{R}^N_+)}\leq
		C\varepsilon.
		\]
		This, together with \eqref{esta1} implies
		$$
		\|u-v\|_{E}\leq\|u-u_1\|_{E}+\|u_1-v\|_{E}\leq\varepsilon+C\varepsilon,
		$$
		which completes the proof.
	\end{proof}

Let 
\[
S_\lambda(p,q)=\inf_{u\in C_0^\infty(\mathbb{R^N})\setminus\{0\}}\frac{
		\displaystyle\int_{\mathbb{R^N_+}}|\nabla u|^p \mathrm{d}  x+\lambda\int_{\mathbb{R}^{N-1}}|u|^p \mathrm{d}  x'}{\left(\displaystyle\int_{\mathbb{R}^{N-1}}|u|^q \mathrm{d}  x'\right)^{p/q}}.
\]

\begin{lemma}\label{111}  Assume that $1<p<N$. Then, 
	\[
		\lim_{\lambda\rightarrow0^+}S_\lambda(p,q)=0,
	\]
	whenever $p<q<p_*$ and 
	\begin{equation}
		\label{Comportamento constante-critica}
		\lim_{\lambda\rightarrow 0^+}S_\lambda(p,p_*)= S_0(p,p_*),
	\end{equation}
	where $S_0(p,p_*)$ is the best constant of the Sobolev trace embedding $\mathcal{D}^{1,p}(\mathbb{R}^N_+)\hookrightarrow L^{p^*}(\mathbb{R}^{N-1}).$
\end{lemma}

\begin{proof}
	Since $S_\lambda(p,q)$ is increasing, we can assume by contradiction that there exists a constant $C_0>0$ such that  $S_\lambda(p,q)>C_0>0$ for all $\lambda>0$. Thus, 
	$$
	C_0\left(\displaystyle\int_{\mathbb{R}^{N-1}}|u|^q \mathrm{d}  x'\right)^{p/q}\leq\left(\displaystyle\int_{\mathbb{R^N_+}}|\nabla u|^p \mathrm{d}  x+\lambda\int_{\mathbb{R}^{N-1}}|u|^p \mathrm{d}  x'\right),\quad \forall u\in C_0^\infty(\mathbb{R^N}).
	$$
	Now taking the limit as $\lambda\rightarrow0^+$ we obtain the inequality
	$$
	C_0\left(\displaystyle\int_{\mathbb{R}^{N-1}}|u|^q \mathrm{d}  x'\right)^{p/q}\leq\left(\displaystyle\int_{\mathbb{R^N_+}}|\nabla u|^p \mathrm{d}  x\right),\quad \forall u\in C_0^\infty(\mathbb{R^N}),
	$$
	which is false by using a standard scaling argument whenever $p<q<p_*$. To prove \eqref{Comportamento constante-critica} we observe that by definition of $S_0(p,p_*)$ we have that 
	$
	S_0(p,p_*)\leq S_\lambda(p,p_*)
	$
	for all $\lambda>0$. Since $S_\lambda(p,p_*)$ is a monotone function in $\lambda $ there exists 
	$$
	\lim_{\lambda\rightarrow0^+} S_\lambda(p,p_*)=L
	$$ 
	and hence $S_0(p,p_*)\leq\lim_{\lambda\rightarrow0^+} S_\lambda(p,p_*)$. Now suppose, by contradiction that $S_0(p,p_*)<L$. Then, for $\varepsilon=(L-S_0(p,p_*)/2$ there exists $\lambda_0>0$ such that for all $\lambda\in(0,\lambda_0)$ we have
	$$
	-\frac{L-S_0(p,p_*)}{2}=-\varepsilon<S_\lambda(p,p_*)-L<\varepsilon
	$$
	that is, 
	$$
	\frac{S_0(p,p_*)+L}{2}\leq S_\lambda(p,p_*)\leq\frac{
		\displaystyle\int_{\mathbb{R^N_+}}|\nabla u|^p \mathrm{d}  x+\lambda\int_{\mathbb{R}^{N-1}}|u|^p \mathrm{d}  x'}{\left(\displaystyle\int_{\mathbb{R}^{N-1}}|u|^q \mathrm{d}  x'\right)^{p/q}}
	$$
	taking the limit in the last inequality, we obtain
	
	$$
	S_0(p,p_*)<\frac{S_0(p,p_*)+L}{2}\leq S_\lambda(p,p_*)\leq\frac{
		\displaystyle\int_{\mathbb{R^N_+}}|\nabla u|^p \mathrm{d}  x}{\left(\displaystyle\int_{\mathbb{R}^{N-1}}|u|^q \mathrm{d}  x'\right)^{p/q}}
	$$
	and this contradicts the definition of $S_0(p,p_*)$. Therefore, we have that $S_0(p,p_*)=\lim_{\lambda\rightarrow0^+}S_\lambda(p,p_*)$.
\end{proof}

	\section{Existence of a ground state} \label{existence}
	This section is devoted to proving Theorem~\ref{T-Existence}. Fixed $0<r<\rho$ let us consider the balls centered in $y\in\mathbb{R}^{N-1}$ and   $(y,0)\in \partial \mathbb{R}^N_+$  with radius $\rho,$
	\[
	\quad\Gamma_\rho(y)=\left\{ x\in\mathbb{R}^{N-1}\text{ : }|x-y|<\rho \right\}\quad\text{and}\quad B_{\rho}^+(y,0)=\left\{ z\in \mathbb{R}_{+}^{N}\text{ : } |z-(y,0)|<\rho  \right\}.
	\]
	For the particular case which $y=0$ we use the notation $B_{\rho}^+=B_{\rho}^+(0)$, $\Gamma_\rho=\Gamma_\rho(0)$ and, let $B_{r}^+\subset B_{\rho}^+$, $\Gamma_r\subset\Gamma_\rho$ be concentric balls.

 \medskip

	\begin{lemma}[Friedrichs]\label{04}
		For all $u\in W^{1,p}(B_1^+)$ there exists a constant $C=C(N,p)>0$ such that 
		\begin{equation}\label{estima}
		 \Vert u \Vert_{L^p(B_1^+)}\leq C\left(\Vert \nabla u \Vert_{L^p(B_1^+)}^p + \Vert u \Vert_{L^p(\Gamma_1)}^p\right)^{1/p}.
		\end{equation}
	\end{lemma}
	
	\begin{proof}
		Suppose by contradiction that \eqref{estima} is false. Thus, for each $k\in\mathbb{N}$, there exist $u_k\in W^{1,p}(B_1^+)$ such that 
		\[
		\Vert u_k \Vert_{L^p(B_1^+)}\geq k\left(\Vert \nabla u_k \Vert_{L^p(B_1^+)}^p + \Vert u_k \Vert_{L^p(\Gamma_1)}^p\right)^{1/p}.
		\]
		Setting 
		\[
		v_k=\frac{u_k}{\left\Vert u_k \right\Vert_{L^p(B_{1}^+)}},
		\]
		we have
		\[
		\quad\Vert v_k \Vert_{L^p(B_1^+)}=1,\quad\Vert v_k \Vert_{L^p(\Gamma_1)}^p\leq\frac{1}{k},\quad\text{and}\quad\Vert\nabla v_k\Vert_{L^p(B_1^+)}\leq\frac{1}{k}.
		\]
		In particular, $(v_k)$ is bounded in $W^{1,p}(B_1^+)$ and, up to a subsequence, we have $v_{k}\rightarrow v$ strongly in $L^p(B_1^+)$ and hence $\Vert v \Vert_{L^p(B_1^+)}=1$. Moreover, by the trace embedding, $v_{k}\rightarrow v $ strongly in $L^q(\Gamma_1)$ with $q\in [p,p_*)$. From $\Vert v_k \Vert_{L^p(\Gamma_1)}^p\leq 1/k$, we get $v\equiv 0$ on $\Gamma_1$.
		
		\noindent Using integration by parts, for all $\phi\in C_{c}^{\infty}(B_1^+),$ we get
		\[
		\int_{B_1^+}v\phi_{x_i}\,\mathrm{d}x=\lim_{k_j\rightarrow +\infty}\int_{B_1^+} v_{k_j}\phi_{x_i}\,\mathrm{d}x=-\lim_{k_j\rightarrow +\infty}\int_{B_1^+}v_{k_j,x_i}\phi\,\mathrm{d}x=0.
		\]
		Consequently, $v\in W^{1,p}(B_1^+)$ and $\nabla v=0$ almost everywhere in $B_1^+$, which implies that $v$ is constant. Since $v\equiv 0$ on $\Gamma_1$, we have $v\equiv 0$ in $B_1^+$ which gives a contradiction with $\Vert v \Vert_{L^p(B_1^+)}=1$.
		\end{proof}
	
	\begin{lemma}\label{05}
		Let $q\in [p,p_*]$ and $y\in\mathbb{R}^{N-1}$. Then, there exists a constant $C=C(N,q)>0$ such that
		\[
		\Vert u \Vert_{L^q(\Gamma_1(y))}\leq C\left( \Vert\nabla u\Vert_{L^p(B_1^+(y))}^{p}+\Vert u \Vert_{L^p(\Gamma_1(y))}^{p} \right)^{1/p},\quad \forall u\in E.
		\]
	\end{lemma}
	
	\begin{proof}
		We can use the Sobolev trace embedding $W^{1,p}(B_1^+) \hookrightarrow L^q(\Gamma_1)$ and Lemma \ref{04} to obtain
		\[
		\Vert u \Vert_{L^q(\Gamma_1)}\leq C\Vert u \Vert_{W^{1,p}(B_1^+)}\leq C\left(\Vert \nabla u \Vert_{L^p(B_1^+)}^p + \Vert u \Vert_{L^p(\Gamma_1)}^p\right)^{},
		\]
		which is the desired inequality.
	\end{proof}
	
	Since $p<q<p_*$, it is standard to see that the  functional $I$ has the mountain pass structure on the Sobolev space $E$. Thus, the minimax level
	\[
	c_q(\mathbb{R}_{+}^{N})=\inf_{\gamma\in\Gamma}\max_{t\in [0,1]}I(\gamma(t)),
	\]
	is positive, where $\Gamma:=\left\{ \gamma\in C([0,1],E)\text{ : }\gamma(0)=0\text{ and }I(\gamma(1))<0 \right\}$ (see \cite{Vv2003} and references therein). Therefore, there exists a Palais-Smale (in short (PS)) sequence for the functional $I$ at the level $c_q(\mathbb{R}_{+}^{N})$ (see for instance \cite[Theorem 2.9]{willen} for more details). Precisely, there exists $(u_n)\subset E$ such that
	\[
	I(u_n)\rightarrow c_q(\mathbb{R}_{+}^{N})\quad\text{ and }\quad I^\prime(u_n)\rightarrow 0.
	\]
	
	\begin{lemma}\label{16}
		If $(u_n)\subset E$ is a $(PS)$ sequence at the minimax level $c_q(\mathbb{R}_{+}^{N})$, then $(u_n)$ is bounded in $E$. Moreover, for $n$  sufficiently large,
		\[
		\Vert u_{n}^+ \Vert_{L^q(\mathbb{R}^{N-1})}\geq \frac{pq}{q-p}\frac{c_q(\mathbb{R}_{+}^{N})}{2}.
		\]
	\end{lemma}
	
	\begin{proof}
		A straightforward computation shows that 
		\[
		\left( \frac{q}{p}-1 \right)\Vert u_n \Vert_{E}^{p}=qI(u_n)-I^\prime(u_n)u_n\leq qc_q(\mathbb{R}_{+}^{N}) + \Vert u_n \Vert_E+1
		\]
		and hence  $(u_n)$ is bounded in $E$. Thus, $I^\prime(u_n)u_n\rightarrow 0$ and consequently for $n$ sufficiently large it holds
		\[
		\frac{c_q(\mathbb{R}_{+}^{N})}{2}\leq I(u_n)-\frac{1}{p}I^\prime(u_n)u_n=\left( \frac{1}{p}-\frac{1}{q} \right)\Vert u_n^+ \Vert_{L^{q}(\mathbb{R}^{N-1})}^{q},
		\]
		which gives the desired estimate.
	\end{proof}
	
	\begin{lemma}\label{17}
		Let $(u_n)\subset E$ be a (PS) sequence. Then, there exists $C=C(N,p)>0$ such that 
		\begin{equation}\label{Lions}
			\sup_{y\in\mathbb{R}^{N-1}}\int_{\Gamma_1(y)}(u_{n}^+)^p\,\mathrm{d}x'\geq C.
		\end{equation}
	\end{lemma}
	\begin{proof}
		Let $\theta\in (0,1)$ be fixed  such that $1/q=(1-\theta)/p+\theta/ p_*$. By interpolation and Lemma \ref{05} we have 
		\begin{equation*}
	    \begin{aligned}
		   \Vert u_n^+ \Vert_{L^q(\Gamma_1(y))}^q 
		   & \leq \Vert u_n^+ \Vert_{L^p(\Gamma_1(y))}^{(1-\theta) q}\Vert u_n^+ \Vert_{L^{p_*}(\Gamma_1(y))}^{\theta q} \\
		   & \leq C\Vert u_n^+ \Vert_{L^p(\Gamma_1(y))}^{(1-\theta)q }\left( \Vert\nabla u_n \Vert_{L^p(B_1^+(y))}^{p}+\Vert u_n \Vert_{L^p(\Gamma_1(y))}^{p} \right)^{\theta q/p}.        
		    \end{aligned}
		\end{equation*}
		Now, we split the proof into two cases:\\
		\noindent \textbf{Case 1:} If $p<\overline{q}:=p(N-2+p)/(N-1)\leq q<p_*$. In this case, $\theta q\geq p$ and hence 
		\begin{multline*}
		    \Vert u_n^+ \Vert_{L^q(\Gamma_1(y))}^q \leq \\ C\left(\sup_{y\in \mathbb{R}^{N-1}}\int_{\Gamma_1(y)}(u_n^+)^p\mathrm{d}x'\right)^{
		\frac{(1-\theta)q}{p}}\Vert u_n \Vert^{\theta q-p}_E \left( \Vert\nabla u_n \Vert_{L^p(B_1^+(y))}^{p}+\Vert u_n \Vert_{L^p(\Gamma_1(y))}^{p} \right).
		\end{multline*}
		Covering $\mathbb{R}^{N-1}$ by a family of balls $\left(\Gamma_1(y)\right)_{y\in\mathbb{R}^{N-1}}$ such that each point of $\mathbb{R}^{N-1}$ is contained in at most $N$ balls. Summing up this inequality over this family, we get
		\[
		\begin{alignedat}{2}
		\Vert u_n^+ \Vert_{L^q(\mathbb{R}^{N-1})}^q & \leq C\left(\sup_{y\in \mathbb{R}^{N-1}}\int_{\Gamma_1(y)}(u_n^+)^p\mathrm{d}x'\right)^{\frac{(1-\theta )q}{p}}\Vert u_n \Vert^{\theta q}_E.
		\end{alignedat}
		\]
		This, together with the fact that $(u_n)$ is bounded and Lemma~\ref{16} implies that \eqref{Lions} holds. 
		
		\noindent \textbf{Case 2:} If $p<q<\overline{q}\leq p_*$. In this case, $q=p\tau+\overline{q}(1-\tau)$ for $\tau=\overline{q}(q-p)/q(\overline{q}-p)\in(0,1)$ and using interpolation again we obtain
		\[
		\Vert u_n^+ \Vert_{L^q(\Gamma_1(y))}^q\leq \Vert u_n^+ \Vert_{L^p(\Gamma_1(y))}^{\tau p}\Vert u_n^+ \Vert_{L^{\overline{q}}(\Gamma_1(y))}^{(1-\tau) \overline{q}}.
		\]
		By a simple calculation we have $\tau \overline{q}/p\geq1$, and hence following as in the first case,
		\begin{multline*}
		    \Vert u_n^+ \Vert_{L^q(\mathbb{R}^{N-1})}^q \leq \\ C\left(\sup_{y\in \mathbb{R}^{N-1}}\int_{\Gamma_1(y)}(u_n^+)^p\mathrm{d}x'\right)^{\frac{(1-\tau ) \overline{q}}{p}}\Vert u_n \Vert^{\tau \overline{q}-p}_E \left( \Vert\nabla u_n \Vert_{L^p(B_1^+(y))}^{p}+\Vert u_n \Vert_{L^p(\Gamma_1(y))}^{p} \right)
		\end{multline*}
which again implies that (3.2) holds and the proof is complete.
\end{proof}

Next, we introduce the Nehari manifold associated to \eqref{P1} to obtain a ground state solution for \eqref{P1},
	\[
	\mathcal{N}:=\left\{ u\in E\setminus \left\{ 0 \right\}: I^\prime (u)u=0 \right\}.
	\]
	Since the set of nontrivial critical points of $I$ is contained into $\mathcal{N}$, we define the least energy level as $c_\mathcal{N}:=\inf_{u\in\mathcal{N}}I(u)$. We recall that a nontrivial weak solution $u$ of \eqref{P1} is a ground state if $c_\mathcal{N}=I(u)$.
	
	\begin{proof}[Finalizing the proof of Theorem \ref{T-Existence}] Let $(u_n)\subset E$ be a (PS) sequence for the functional $I$ at the level $c_q(\mathbb{R}_{+}^{N})$. By using Lemma \ref{17}, there exists a sequence $(y_n)\subset \mathbb{R}^{N-1}$ such that
		\[
		\int_{\Gamma(y_n)}|u_n^+|^p\,\mathrm{d}x'\geq \frac{C}{2}>0.
		\]
		Defining $w_n(x)=u_n(x+y_n),$ we have
		\begin{equation}\label{PStransladada}
			\int_{\Gamma(0)}|w_n^+|^p\,\mathrm{d}x'\geq \frac{C}{2},\quad I(w_n)\rightarrow c_q(\mathbb{R}_{+}^{N})\quad \text{and}\quad I^\prime(w_n)\rightarrow 0.
		\end{equation}
		As stated in Lemma \ref{16}, $(w_n)$ is bounded, and up to a subsequence $w_n\rightharpoonup w$  weakly in $E, $
		$w_n\rightarrow w$ in $L^s_{\mathrm{loc}}(\mathbb{R}_{+}^{N})$ for all $1 < p \leq s  <p^*= Np/(N-p)$
		and $w_n\rightarrow w$ in $L^q_{\mathrm{loc}}(\mathbb{R}^{N-1})$ for all $1 < p< q < p_*:=p(N-1)/(N-p)$. Hence, by \eqref{PStransladada}, $w$ is nontrivial. 
		\medskip
		
		We  can see that   $ ( |w_n|^{r-2} w_n )_{n \in \mathbb{N}}$ converges weakly to $|w|^{r-2} w$ in $ L^{\frac{r}{r-1}} (\mathbb{R}^{N-1})$ for any $p \leq r < p_*$, that is, for any $\varphi \in L^{r} (\mathbb{R}^{N-1})$ we have 
		\begin{equation}\label{0001}
		 \int_{\mathbb{R}^{N-1}} \vert w_n \vert^{r-2} w_n \varphi\,\mathrm{d}x'
		 \rightarrow
		 \int_{\mathbb{R}^{N-1}}\vert  w\vert^{r-2} w  \varphi\,\mathrm{d}x'.
		\end{equation}

		\begin{claim} For any $\varphi \in E$ we have 
		\begin{equation}\label{massaranduba}
		 \int_{\mathbb{R}_{+}^{N}} \vert \nabla w_n \vert^{p-2}\nabla w_n\nabla \varphi\,\mathrm{d}x
		 \rightarrow
		 \int_{\mathbb{R}_{+}^{N}}\vert \nabla w\vert^{p-2}\nabla w\nabla \varphi\,\mathrm{d}x   .
		\end{equation}
		\end{claim}
		
		Since the sequence $ (\vert \nabla w_n \vert^{p-2}\nabla w_n )_{n \in \mathbb{N}}$ is bounded in $ (L^{\frac{p}{p-1}} (\mathbb{R}_{+}^{N}) )^N $, up to a subsequence we may assume that there exists a function $ V \in (L^{\frac{p}{p-1}} (\mathbb{R}_{+}^{N})^N $ such that 
		\begin{equation}\label{0005}
		 \vert \nabla w_n \vert^{p-2}\nabla w_n  \rightharpoonup  V  \quad \text{weakly in} \quad  (L^{\frac{p}{p-1}} (\mathbb{R}_{+}^{N}))^N.
		\end{equation}
		From \eqref{PStransladada}, we have $ \forall \varphi\in E,$
		\begin{equation}\label{02}
		    \left|	\int_{\mathbb{R}_{+}^{N}}\vert \nabla w_n \vert^{p-2}\nabla w_n\nabla \varphi\,\mathrm{d}x-\int_{\mathbb{R}^{N-1}}\left[\vert w_n \vert^{q-2}w_n-\lambda \vert w_n \vert^{p-2}w_n\right]\varphi\,\mathrm{d}x'\right| \leq \varepsilon_n \|\varphi\|_E,
		\end{equation}
		where $\varepsilon_n \rightarrow 0$ as $n \rightarrow \infty$.
	
		Let us consider a bump function $\psi_R\in C_0^\infty(\mathbb{R}^N,[0,1])$ 
		such that $\psi_R\equiv 1$ on the compact subset $K_R=\overline{B_R(0)\cap\mathbb{R}^N_+}$ of $\mathbb{R}^N_+$ and $T_\eta(s)=s$ if $|s|\leq\eta$, and $T_\eta(s)=\eta s/|s|$ if $|s|\geq\eta$. Thus, from \eqref{02} we obtain
		\begin{equation}\label{0004}
		\begin{alignedat}{2}
		    \left|	\int_{\mathbb{R}_{+}^{N}}\psi_R \vert \nabla w_n \vert^{p-2}\nabla w_n \right. & \nabla T_\eta(w_n-w)\,\mathrm{d}x+\int_{\mathbb{R}_{+}^{N}}T_\eta(w_n-w)\vert \nabla w_n \vert^{p-2}\nabla w_n\nabla\psi_R\,\mathrm{d}x  \\ &\left. -\int_{\mathbb{R}^{N-1}}\left[\vert w_n \vert^{q-2}w_n-\lambda \vert w_n \vert^{p-2}w_n\right] \psi_RT_\eta(w_n-w)\,\mathrm{d}x' \right|\\
		    &\leq \varepsilon_n \|\psi_RT_\eta(w_n-w)\|_E,
		\end{alignedat}
		\end{equation}
Since
\begin{equation*}
\int_{\mathbb{R}_{+}^{N}} \psi_R \vert \nabla w \vert^{p-2}\nabla w  \nabla T_\eta(w_n-w)\,\mathrm{d}x =\int_{|w_n-w|\leq \eta} \psi_R \vert \nabla w \vert^{p-2}\nabla w  \nabla (w_n-w)\,\mathrm{d}x
\end{equation*}
and	$w_n\rightharpoonup w$  weakly in $E$, we have
	\begin{equation}\label{Camboinha-1}
	   \lim_{n \rightarrow \infty } \int_{\mathbb{R}_{+}^{N}} \psi_R \vert \nabla w \vert^{p-2}\nabla w  \nabla T_\eta(w_n-w)\,\mathrm{d}x =0.
	\end{equation}
	Since the sequence $|\psi_R T_\eta(w_n-w)|\leq \psi_R |w_n-w|$ is bounded in $ L^{\frac{r}{r-1}} (\mathbb{R}^{N-1})$ for any $p \leq r < p_*$, from \eqref{0001} we have 
		\begin{equation}\label{0003}
		\lim_{n \rightarrow \infty}    \int_{\mathbb{R}^{N-1}}\left[\vert w_n \vert^{q-2}w_n-\lambda \vert w_n \vert^{p-2}w_n\right] \psi_RT_\eta(w_n-w)\,\mathrm{d}x' =0.
		\end{equation}
	Using H\"older inequality,
	\begin{multline*}
	    \left|\int_{\mathbb{R}_{+}^{N}}T_\eta(w_n-w)\vert \nabla w_n \vert^{p-2}\nabla w_n\nabla\psi_R\,\mathrm{d}x\right|\leq\\ \left(\int_{K_R}|T_\eta(w_n-w)|^{p}\right)^{1/p}\left(\int_{K_R}\vert \nabla w_n \vert^{p}\right)^{(p-1)/p}
\leq\\ C\left(\int_{K_R}|w_n-w|^{p}\right)^{1/p}.
	\end{multline*}
Thus, 
\begin{equation}\label{0002}
    \lim_{n\rightarrow +\infty}\int_{\mathbb{R}_{+}^{N}}T_\eta(w_n-w)\vert \nabla w_n \vert^{p-2}\nabla w_n\nabla\psi_R\,\mathrm{d}x=0.
\end{equation}
Since $\psi_RT_\eta(w_n-w)$ is a bounded sequence in $E$, from \eqref{0004}, \eqref{0003} and \eqref{0002} we have
	\begin{equation}\label{Formosa}
	   \lim_{n \rightarrow \infty } \int_{\mathbb{R}_{+}^{N}} \psi_R \vert \nabla w_n \vert^{p-2}\nabla w_n  \nabla T_\eta(w_n-w)\,\mathrm{d}x =0.
	\end{equation}
	From \eqref{Camboinha-1} and \eqref{Formosa} we obtain 
	\begin{equation}\label{Manaira}
	   \lim_{n \rightarrow \infty } \int_{\mathbb{R}_{+}^{N}} 
	   \psi_R \left[\vert \nabla w_n \vert^{p-2}\nabla w_n - \vert \nabla w \vert^{p-2}\nabla w \right]\nabla T_\eta(w_n-w)\,\mathrm{d}x =0.
	\end{equation}
\begin{claim}
 $\nabla w_n\rightarrow \nabla w\quad \text{almost everywhere in }\mathbb{R}_+^N$
\end{claim}	

Defining the nonnegative function  $\Psi_n(x):=\left[|\nabla w_n|^{p-2}\nabla w_n-|\nabla w|^{p-2}\nabla w\right]\nabla ( w_n-w)$ and fixing $\theta\in(0,1)$, using \eqref{Manaira} and H\"older's inequality we see that  
		\begin{align*}
			\int_{K_R}\Psi_n^\theta \, \mathrm{d}  x&=\int_{K_R\cap\{|w_n-w|\leq\eta\}}\Psi_n^\theta \, \mathrm{d}  x+\int_{K_R\cap\{|w_n-w|\geq\eta\}}\Psi_n^\theta \, \mathrm{d}  x\\
			&\leq o_n(1)|K_R|^{1-\theta}+\left(\int_{K_R\cap\{|w_n-w|>\eta\}}\Psi_n \, \mathrm{d}  x\right)^\theta |{K_R\cap\{|w_n-w|>\eta\}}|^{1-\theta}.
		\end{align*}
		Since $\eta>0$ is fixed, by the assumptions, $|{K_R\cap\{|w_n-w|\geq\eta\}}|\rightarrow 0$ as $n\rightarrow+\infty$. As a consequence,  $\Psi_n^\theta$ converges strongly to zero in $L^1(K_R)$. From \cite[Lemma 2.2]{Simon} there exists $C_0>0$ such that 
		$$
		\Psi_n(x)\geq C_0\left\{
		\begin{aligned}
		&|\nabla w_n(x)-\nabla w(x)|^p\quad&\mbox{if}&\ \ p\geq2,\\
		&\frac{|\nabla w_n(x)-\nabla w(x)|^2}{(1+|\nabla w_n(x)|+|\nabla w(x)|)^{2-p}}\quad&\mbox{if}&\ \ 1<p<2.
		\end{aligned}
		\right.
		$$
		Therefore, the claim holds. From \eqref{0005}
		and $\nabla w_n\rightarrow \nabla w\quad \text{almost everywhere in }\mathbb{R}_+^N$, we can conclude that $V=\vert \nabla w \vert^{p-2}\nabla w$ almost everywhere in $\mathbb{R}_+^N$ and \eqref{massaranduba} holds..
		\begin{claim}
		 $w\in\mathcal{N}$ 
		\end{claim}
				From \eqref{02} and \eqref{massaranduba}, since $w_n\rightarrow w$ in $L^r_{\mathrm{loc}}(\mathbb{R}^{N-1})$ for any $1\leq p\leq r <p^*$ and using Lemma \ref{0006}, we obtain 
		$	I^\prime(w)\varphi=0, \; \forall \varphi\in E,
		$
        which, together with the fact that $w$ is nontrivial, proves the claim. 
        
        Since
		\[
		I(w_n)-\frac{1}{p}I^\prime(w_n)w_n=\left( \frac{1}{p}-\frac{1}{q} \right)\Vert w_n^+ \Vert_{L^q(\mathbb{R}^{N-1})}^q\rightarrow c_q(\mathbb{R}_{+}^{N}),
		\]
		and using the weak lower semi-continuity of the norm,  we get
		$
		I(w)\leq c_q(\mathbb{R}_{+}^{N}).
		$

        To complete the proof, we use \cite[Theorem 1.8]{Vv2003}, which proves that the mountain-pass level gives the least energy level.
        Therefore, $c_q(\mathbb{R}_{+}^{N})\leq I(w)$ and hence $w$ is a ground state solution of \eqref{P1}.
\end{proof}

	\section{Regularity and decay}\label{reg}
	This section is devoted to proving Theorem~\ref{T-Regularity-Decay}. First, we use the Moser iteration method to obtain a $L^\infty$ estimate.

	\begin{lemma}[Boundedness]\label{13}
		If $u\in E$ is a weak solution of Problem \eqref{P1} then, in the trace sense $u\vert_{\mathbb{R}^{N-1}}\in L^\infty(\mathbb{R}^{N-1})$ and  $u\in L^\infty(\mathbb{R}_{+}^{N})$. 
	\end{lemma}
	\begin{proof}
	For each $n\in\mathbb{N}$ and $\beta>1$ we define $\phi_n=uu_{n}^{p(\beta-1)}$, where $u_n=\min\left\{ |u|,n \right\}$. Taking $\phi_n$ as a test function in \eqref{Def1} one has
		\[
		\int_{\mathbb{R}_{+}^{N}}|u_n|^{p(\beta-1)}|\nabla u|^p\,\mathrm{d}x\leq \int_{\mathbb{R}^{N-1}}|u|^{q}|u_{n}|^{p(\beta-1)}\,\mathrm{d}x'.
		\]
		By using the trace embedding and H\"older's inequality with exponents $p_*/(q-p)$ and $p_*/(p_*-q+p)$ we have
		\begin{multline*}
		    \left[\int_{\mathbb{R}^{N-1}}\left( |u u_{n}|^{(\beta-1)}\right)^{p_*}\,\mathrm{d}x'\right]^\frac{p}{p_*}\leq \\
      C_{N,p}\int_{\mathbb{R}_N^+}\left\vert \nabla (uu_n^{(\beta-1)})\right\vert^p \mathrm{d}x \leq \\
      \beta^pC_{N,p}\int_{\mathbb{R}_{+}^{N}}|u_n|^{p(\beta-1)}|\nabla u|^p  \,\mathrm{d}x \leq \\
      \beta^pC_{N,p} \int_{\mathbb{R}^{N-1}}|u|^{q}|u_{n}|^{p(\beta-1)}\,\mathrm{d}x' \leq \\
      \beta^pC_{N,p}\left[\int_{\mathbb{R}^{N-1}} |u|^{p_*} \,\mathrm{d}x'\right]^{\frac{q-p}{p_*}}\left[ \int_{\mathbb{R}^{N-1}} |u|^{\frac{pp_*\beta}{p_*-q+p}} \,\mathrm{d}x' \right]^\frac{p_*-q+p}{p_*}.
		\end{multline*}
	By choosing $\beta=(p_*-q+p)/p$ and $\alpha=pp_*/(p_*-q+p)$ follows
	\[
	\begin{alignedat}{2}
	\left[\int_{\mathbb{R}^{N-1}}\left( |u u_{n}|^{(\beta-1)}\right)^{p_*}\,\mathrm{d}x'\right]^\frac{p}{p_*}&\leq \beta^pC_{N,p}\left[\int_{\mathbb{R}^{N-1}} |u|^{p_*} \,\mathrm{d}x'\right]^{\frac{q-p}{p_*}}\left[ \int_{\mathbb{R}^{N-1}} |u|^{\alpha\beta} \,\mathrm{d}x' \right]^{\frac{p\beta}{\alpha\beta}}\\
	&\leq \beta^pC_{N,p}\Vert u \Vert^{q-p}_{L^{p_*}(\mathbb{R}^{N-1})}\Vert u \Vert^{p\beta}_{L^{\alpha\beta}(\mathbb{R}^{N-1})}.
	\end{alignedat}
	\]
	Fatou's Lemma implies
	\begin{equation}\label{14}
	\Vert u \Vert_{L^{p_*\beta}(\mathbb{R}^{N-1})}\leq \left( \beta^pC_{N,p}\Vert u \Vert^{q-p}_{L^{p_*}(\mathbb{R}^{N-1})}\right)^{1/(p\beta)}	\Vert u \Vert_{L^{\alpha\beta}(\mathbb{R}^{N-1})}.
	\end{equation}
	Taking $\beta_0=\beta$ and $\beta_{1}\alpha=p_*\beta_0$ in \eqref{14} we obtain
	\[
	\begin{alignedat}{2}
	\Vert u \Vert_{L^{p_*\beta_1}(\mathbb{R}^{N-1})} &\leq \left( \beta_1^pC_{N,p}\Vert u \Vert^{q-p}_{L^{p_*}(\mathbb{R}^{N-1})}\right)^{1/(p\beta_1)}	\Vert u \Vert_{L^{\alpha\beta_1}(\mathbb{R}^{N-1})}\\
	&=\left( \beta_1^pC_{N,p}\Vert u \Vert^{q-p}_{L^{p_*}(\mathbb{R}^{N-1})}\right)^{1/(p\beta_1)}	\Vert u \Vert_{L^{p_*\beta_0}(\mathbb{R}^{N-1})}\\
	&\leq \left( C_{N,p}\Vert u \Vert^{q-p}_{L^{p_*}(\mathbb{R}^{N-1})}\right)^{1/(p\beta_1)+1/(p\beta_0)} \beta_1^{1/\beta_1}\beta_0^{1/\beta_0} \Vert u \Vert_{L^{p_*}(\mathbb{R}^{N-1})}
	\end{alignedat}
	\]
	Defining $\beta_m=(p_*/\alpha)^m\beta_0,$ by a similar argument, we can prove that
	\begin{multline}\label{Camboinha}
	    	\Vert u \Vert_{L^{p_*\beta_m}(\mathbb{R}^{N-1})}\leq \\ \left( C_{N,p}\Vert u \Vert^{q-p}_{L^{p_*}(\mathbb{R}^{N-1})}\right)^{\frac{1}{p\beta_0}\sum_{i=0}^{m}\left(\frac{p_*}{\alpha}\right)^{-i}} \beta_0^{\frac{1}{\beta_0}\sum_{i=0}^{m}\left(\frac{p_*}{\alpha}\right)^{-i}} \left(\frac{p_*}{\alpha}\right)^{\frac{1}{\beta_0}\sum_{i=0}^{m}\left(\frac{p_*}{\alpha}\right)^{-i}} \Vert u \Vert_{L^{p_*}(\mathbb{R}^{N-1})}.
	\end{multline}
	Since $p_*/\alpha>1$, we have  $1/\beta_0\sum_{i=0}^{m}\left(p_*/\alpha\right)^{-i}\rightarrow p/(p_*-q)$ as $m\rightarrow +\infty$. Hence, taking $m\rightarrow +\infty$ in \eqref{Camboinha}, it follows that
	\[
	\Vert u \Vert_{L^{\infty}(\mathbb{R}^{N-1})}\leq \left(C_{N,p}\Vert u \Vert^{q-p}_{L^{p_*}(\mathbb{R}^{N-1})}\right)^{1/(p_*-q)} \beta_0^{p/(p_*-q)} \left(\frac{p_*}{\alpha}\right)^{p/(p_*-q)} \Vert u \Vert_{L^{p_*}(\mathbb{R}^{N-1})}.
	\]
    Thus, $u\in L^{\infty}(\mathbb{R}^{N-1})$ and let  $M>0$ such that $\|u\|_{L^\infty(\mathbb{R}^{N-1})}\leq M$. 
    For each $k\in \mathbb{N}$, let us consider the superlevel set
		\[
		\Omega_k=\left\{ z=(x^\prime,x_N)\in\overline{\mathbb{R}_{+}^{N}}\text{ : } |u(z)|>k \right\}.
		\]
		Observe that $\Omega_{k}$ has finite Lebesgue measure since $u\in L^{p^*}(\mathbb{R}_{+}^{N})$.
		 Thus, if $k>M$ taking 
		\[
		\psi(z)=\left\{
		\begin{alignedat}{2}
		u(z)-k&\quad \text{if}\quad z\in\Omega_k,\\
		0&\quad \text{if}\quad z\in \mathbb{R}_{+}^{N}\setminus\Omega_k,
		\end{alignedat}
		\right.
		\]
		as a test function in \eqref{Def1} infer that 
		\[
		\int_{\Omega_k}\vert \nabla u \vert^p\,\mathrm{d}x=0,
		\]
		which implies that $u$ is constant or $\vert \Omega_k \vert=0$. Therefore,  $u\in L^{\infty}(\mathbb{R}_{+}^{N})$ and this completes the proof. 
	\end{proof}
	
	As a consequence of Lemma \ref{13}, weak solutions of \eqref{P1} belong to class $C^{1,\alpha}_{\mathrm{loc}}(\overline{\mathbb{R}^N_+})$ for some $0<\alpha<1$ (see \cite[Theorem 2]{Liber}). 
	To prove the polynomial decay stated in Theorem~\ref{T-Regularity-Decay}, in our argument, we need the following auxiliary result.
	
	\begin{lemma}[Harnack type inequality]\label{12}
		Let $u$ be a weak solution of Problem \eqref{P1} with $0< u \leq M$ in $B_{3\rho}^{+}$ and $1<\rho<2$. Then, there exist $C=C(N,M)$ and $\theta_0>1$ such that
		\[
		\max_{B_1^+}u+\max_{\Gamma_1}u\leq C\rho^{-(N-1)/\theta_0}\left( \rho^{-1}\Vert u \Vert_{L^{\theta_0}(B_{2\rho}^+)}^{\theta_0}+\Vert u \Vert_{L^{\theta_0}(\Gamma_{2\rho})}^{\theta_0} \right)^{1/\theta_0}.
		\]
		In particular, by Lemma \ref{13} we see that
		\[
		\lim_{|x|\rightarrow +\infty}u(x)=0\quad \text{ for all }\quad x\in\overline{ \mathbb{R}_{+}^{N}}.
		\]
	\end{lemma}
	
	\begin{proof}
		Let $u$ be a weak solution of \eqref{P1} and  $\zeta\in C^1(B_{3\rho})$ satisfying $0\leq \zeta\leq 1$ and $\mathrm{supp}(\zeta)\subset B_{\rho}^{+}$. 
		For $\beta>1$ taking  $\phi=\zeta^pu^\beta$ as test function in \eqref{Def1} we obtain
	\[
	\begin{split}
	      p\int_{\mathbb{R}_{+}^{N}}\zeta^{p-1}u^\beta\vert \nabla u \vert^{p-2}\nabla u\nabla \zeta\,\mathrm{d}x&+\beta\int_{\mathbb{R}_{+}^{N}}\zeta^{p}u^{\beta-1}\vert \nabla u \vert^{p}\,\mathrm{d}x+\int_{\mathbb{R}^{N-1}}\zeta^p u^{p+\beta-1}\,\mathrm{d}x'= \\&\int_{\mathbb{R}^{N-1}}\zeta^p (u^+)^{q+\beta-1}\,\mathrm{d}x',
	\end{split}
	\]
		which implies
		\[
		\beta\int_{B_{\rho}^+}\zeta^pu^{\beta-1}\vert \nabla u \vert^p\,\mathrm{d}x\leq p\int_{B_{\rho}^+}\zeta^{p-1}u^{\beta}\vert \nabla u \vert^{p-1}\vert \nabla \zeta \vert\,\mathrm{d}x + M^{q-p}\int_{\Gamma_\rho}\zeta^pu^{\beta+p-1}\,\mathrm{d}x'.
		\]
		Applying Young's inequality, we conclude that
		\begin{multline*}
		\int_{B_{\rho}^+}\zeta^pu^{\beta-1}\vert \nabla u \vert^p\,\mathrm{d}x\leq \\ \frac{C_{p,q,M}}{\beta-(p-1)\epsilon^\frac{p}{p-1}}\left[ \epsilon^{-p}\int_{B_{\rho}^+}u^{\beta+p-1}\vert \nabla\zeta \vert^p\,\mathrm{d}x +\int_{\Gamma_\rho}\zeta^pu^{\beta+p-1}\,\mathrm{d}x'. \right]    
		\end{multline*}
		Choosing $\beta$ large enough and setting $w=u^s$, where $sp=\beta+p-1$, we have 
		\[
		\frac{1}{s^p}\int_{B_{\rho}^+}\left( \zeta|\nabla w| \right)^p\,\mathrm{d}x\leq C_{p,q,M}\beta^{-1}\left[ \int_{B_{\rho}^+}\left( |\nabla \zeta|w \right)^p\,\mathrm{d}x+\int_{\Gamma_\rho}\left( \zeta w \right)^p\,\mathrm{d}x' \right],
		\]
		which implies
		\begin{multline*}
		\left[ \Vert \zeta\vert\nabla w\vert \Vert_{L^p(B_{\rho}^+)}^p+\Vert \zeta w \Vert_{L^p(\Gamma_\rho)}^p \right]^{1/p}\leq \\C_{p,q,M}\left(1+\beta^{-1}\right)s\left[ \Vert w\vert\nabla\zeta\vert \Vert_{L^p(B_{\rho}^+)}^p + \Vert \zeta w \Vert_{L^p(\Gamma_\rho)}^p\right]^{1/p}.
		\end{multline*}
		Let $r_1,r_2\in\mathbb{R}$ be satisfying  $1\leq r_2<\rho\leq r_1\leq 2$. Taking $\zeta \equiv 1$ in $B_{r_2}$ and $\zeta \equiv 0$ in $\mathbb{R}_{+}^{N}\setminus B_{r_1}$ with $\vert \nabla \zeta \vert\leq 2/(r_1-r_2)$ we achieve
		\begin{equation}\label{06}
			\left[ \Vert\nabla w\Vert_{L^p(B_{r_2}^+)}^p+\Vert  w \Vert_{L^p(\Gamma_{r_2})}^p \right]^{1/p}\leq \frac{sC_{p,q,M,\beta}}{r_1-r_2}\left[ \Vert w \Vert_{L^p(B_{r_1}^+)}^p + \Vert w \Vert_{L^p(\Gamma_{r_1})}^p\right]^{1/p}.
		\end{equation}
		By using Sobolev embedding and Lemma \ref{05} we have
		\begin{equation}\label{07}
			\Vert w \Vert_{L^{p_*}(B_{r_2}^+)}+\Vert w \Vert_{L^{p_*}(\Gamma_{r_2})}\leq C\left(\Vert \nabla w \Vert_{L^p(B_{r_2}^+)}^{p}+ \Vert w \Vert_{L^p(\Gamma_{r_2})}^{p}\right)^{1/p}.
		\end{equation}
		Since $w=u^s$, we can use \eqref{07} and \eqref{06} to obtain
		\begin{equation}\label{08}
			\left( \int_{B_{r_2}^+}\vert u \vert^{sp_*}\,\mathrm{d}x+\int_{\Gamma_{r_2}}u^{sp_*}\,\mathrm{d}x' \right)^{1/p_*}\leq \frac{sC}{r_1-r_2}\left( \int_{B_{r_1}}\vert u \vert^{sp}\,\mathrm{d}x+\int_{\Gamma_{r_1}}\vert u \vert^{sp}\,\mathrm{d}x' \right)^{1/p}.
		\end{equation}
		Let us consider the auxiliary function $\psi: (0,+\infty) \times (0,+\infty)\rightarrow \mathbb{R}$ defined by
		\[
		\psi(t,r)=\left( \int_{B_{r}^+}\vert u \vert^{t}\,\mathrm{d}x+\int_{\Gamma_{r}}\vert u\vert^{t}\,\mathrm{d}x' \right)^{1/t}.
		\]
		Taking the $s$-th root in \eqref{08} it follows
		\begin{equation}\label{09}
			\psi(sp_*,r_2)\leq \left(\frac{sC}{r_1-r_2}\right)^{1/s}\psi(sp,r_1).
		\end{equation}
		Denoting $\theta_0=sp$, $\theta_1=\gamma \theta_0$ and $p\gamma=p_*,$ we can write \eqref{09} as follows  
		\begin{equation}\label{09a}
			\psi(\theta_1,r_2)\leq \left(\frac{sC}{r_1-r_2}\right)^{1/s}\psi(\theta_0,r_1).
		\end{equation}
Defining $\theta_m=\gamma^m\theta_0$ and $ r_m=1+2^{-m},\quad m\in\mathbb{Z}_+$, from \eqref{09a} by induction on $m$ we get
		\begin{equation}\label{10}
			\begin{alignedat}{2}
				\psi(\theta_{m+1},r_{m+1})&\leq \left(\frac{C\gamma^{m+1}\theta_0}{r_m-r_{m+1}}\right)^{p/(\gamma^{m}\theta_0)}\psi(\theta_m,r_m)\\
				&\leq \left( C(2\gamma)^{m+1}\right)^{p/(\gamma^{m}\theta_0)}\psi(\theta_m,r_m)\\
				&=\left(C^\frac{p}{\theta_0}\right)^{1/\gamma^{m}}\left( (2\gamma)^\frac{p}{\theta_0} \right)^{(m+1)/\gamma^{m}}\\
				&\leq \left(C^\frac{p}{\theta_0}\right)^{\sum 1/\gamma^{m}}\left( (2\gamma)^\frac{p}{\theta_0} \right)^{\sum (m+1)/\gamma^{m}}\psi(\theta_0,2).
			\end{alignedat}
		\end{equation}
		Choose $\theta_0$ such that $\theta_m\neq 1$ and observe that $\gamma>1$. We can take the limit as $m\rightarrow +\infty$ in \eqref{10} to obtain
		\[
		\max_{B_1^+}u+\max_{\Gamma_1}u=\psi(+\infty,1)\leq C\psi(\theta_0,2)=\left( \int_{B_{2}^+}\vert u \vert^{\theta_0}\,\mathrm{d}x+\int_{\Gamma_{2}}\vert u\vert^{\theta_0}\,\mathrm{d}x' \right)^{1/\theta_0}.
		\]
		Using the change of variable $x=\overline{x}/\rho$ with $z\in B_{2}^+$ and $x=\overline{x'}/\rho$ with $x'\in \Gamma_2$ we conclude the proof.
	\end{proof}
	
	\begin{proof}[Proof of Theorem~\ref{T-Regularity-Decay}]
		Let $u_\lambda$ be a weak solution of Problem \eqref{P1} and $A>0$ a constant to choose later. Consider $\phi:\overline{ \mathbb{R}_{+}^{N}} \rightarrow \mathbb{R}$ defined by $\phi=(Au_\lambda-v)_+$, where
		\[
		v(x',x_N)=\left( \frac{\mu^{{p}/{2}}}{(\mu+x_N)^2+|x'|^2} \right)^{(N-p)/(2(p-1))}, \quad \mu>0.
		\]
		We can check that $v$ is a positive solution of
		\[
		\left\{
		\begin{alignedat}{2}
		\Delta_pv&=0&\quad&\text{ in }\quad\mathbb{R}_{+}^{N},\\
		\vert \nabla v \vert^{p-2}\frac{\partial v}{\partial \nu}&=\left( \frac{N-p}{p-1} \right)^{p-1}v^{p_*-1}&\quad&\text{ on }\quad\partial\mathbb{R}_{+}^{N}.
		\end{alignedat}
		\right.
		\]
		Using the Harnack inequality proved in Lemma \ref{12}, one can see that $u(x)\rightarrow 0$ as $|x|\rightarrow +\infty$. Thus, we can take $R>0$ large enough such that $u_\lambda^{q-p}(x^\prime,0)<1/2$ if $|x^\prime|\geq R$. Next, choose $A>0$ such that $\phi(x^\prime,0)=0$ if $|x^\prime|\leq R$. Now observe that
		\begin{multline*}
		\int_{|x^\prime|\geq R}\vert \nabla \phi \vert^p\,\mathrm{d}x\leq \\\int_{|x^\prime|\geq R}\left(A^{p-1} |\nabla u|^{p-2}\nabla u-|\nabla v|^{p-2}\nabla v\right)\nabla\phi\,\mathrm{d}x \leq \\\int_{|x^\prime|\geq R}\left( -\left( \frac{N-p}{p-1} \right)^{p-1}v^{p_*-1}+A^{p-1}\left(|u|^{q-1}-\lambda|u|^{p-1}\right) \right)\phi\,\mathrm{d}x'\leq 0.
		\end{multline*}
		Thus $\phi\equiv 0$ in $\overline{\mathbb{R}_{+}^{N}}$, which implies that $u\leq Cv$ in $\overline{\mathbb{R}_{+}^{N}}$, and this yields the desired polynomial decay.
	\end{proof}
	
	\section{Nonexistence}\label{poho}
	In this section, we first prove a Poho\v{z}aev type identity, which is crucial in our argument to prove our nonexistence result state in Theorem~\ref{T0}. Since the seminal paper \cite{Po}, some Poho\v{z}aev type identities have been settled in different contexts, however, always requiring solutions with higher regularity.
	See, for instance, \cite{rossi,PS,Isa} and references therein. 
	
	As it is well known, in general, $C^{1,\alpha}(\overline{\Omega})$ is the optimal regularity for solutions of boundary problems involving the $p$-Laplace operator, see \cite{Liber,Simon}.
	
	To prove a Poho\v{z}aev identity for solutions of \eqref{P1}, we need to impose some more restrictive  assumptions to improve the regularity, or we can use an approximation procedure as in \cite{DMS}. This last one is developed to the Dirichlet boundary conditions problem, but with minor details, it can be performed in the Neumann case. Nevertheless, the boundedness of the domain is required to improve the regularity. This property was successfully explored in \cite{PuSe,ITa} to get optimal $W^{2,2}_{\mathrm{loc}}$ and $W^{2,p}_{\mathrm{loc}}$ regularity respectively, and as a consequence they proved some versions of this identity. 

\medskip
 
	\begin{proposition}\label{Pohozaev} Let $f,g:\mathbb{R}\rightarrow\mathbb{R}$ be continuous functions, and $u\in C^1_{\mathrm{loc}}(\overline{\mathbb{R}^N_+})$ be a weak solution of the problem
		\begin{align*}
			\left\{
			\begin{aligned}
				-\Delta_pu&=f(u)&\quad&\text{ in }\quad\mathbb{R}_{+}^{N},\\
				\vert \nabla u \vert^{p-2}\frac{\partial u}{\partial \nu}&=g(u)&\quad&\text{ on }\quad\mathbb{R}^{N-1},
			\end{aligned}
			\right.
		\end{align*}
		satisfying $|\nabla u|^p\in L^1(\mathbb{R}^N_+)$, $|\nabla u|^{p-1} \in W^{1,1}_{\mathrm{loc}}(\mathbb{R}^N_+)$,  $F(u)\in L^1(\mathbb{R}^N_+)$, and $G(u)\in L^1(\mathbb{R}^{N-1})$  where 
		\[
		F(t)=\int_0^tf(s)ds, \ \ \mbox{ and }\ \ G(t)=\int_0^tg(s)ds.
		\]
		Then, $u$ satisfies the identity
		\[
		\frac{N-p}{p}\int_{\mathbb{R}_{+}^{N}}|\nabla u|^p\,\mathrm{d}x-N\int_{\mathbb{R}^{N}_+}F(u)\,\mathrm{d}x=(N-1)\int_{\mathbb{R}^{N-1}}G(u)\,\mathrm{d}x^\prime.
		\]
	\end{proposition}
	\begin{proof}
		The proof is similar to the one in \cite[Theorem 4.2]{ITa}. Thus, we will only point out the main steps. Let $\Omega_R:=B_R(0)\cap \overline{\mathbb{R}^N_+}$ be fixed. Since $u\in C^1_{\mathrm{loc}}(\mathbb{R}^N_+)\cap W^{2,2}_{\mathrm{loc}}(\mathbb{R}^N_+)$, by formal calculation we have
		\begin{align*}
			{
				\rm{div}}\left((x\cdot\nabla u)|\nabla u|^{p-2}\nabla u\right)&=(x\cdot\nabla u)\Delta_p u+\nabla(x\cdot\nabla u)\cdot|\nabla u|^{p-2}\nabla u\\
			&=(x\cdot\nabla u)\Delta_p u+|\nabla u|^p+\frac{1}{p}\left(x\cdot\nabla (|\nabla u|^{p})\right)
		\end{align*}
		and
		$$
		{\rm{div}}\left(x|\nabla u|^{p}\right)=N|\nabla u|^p+\left(x\cdot\nabla (|\nabla u|^{p})\right)
		$$
		which holds in 
		\[
		\Omega_{R,0}:=\left\{x\in\Omega_R\ :\ |\nabla u(x)|\neq0\right\}.
		\]
		Consequently, 
		\begin{equation}\label{boa}
			{\rm{div}}\left((x\cdot\nabla u)|\nabla u|^{p-2}\nabla u-\frac{1}{p}x|\nabla u|^p\right)=(x\cdot\nabla u)\Delta_p u+\frac{p-N}{p}|\nabla u|^p\quad\mbox{in}\ \ \Omega_{R,0}.
		\end{equation}
		Now, consider $\varphi_{_R}\in C^{1,+}_0(\Omega_{R})$ such that $\varphi_{_R}(x)=1$ if $|x|\leq R/2$, $\varphi_{_R}(x)=0$ if $|x|\geq R$, and $|\nabla \varphi_{_R}(x)|\leq2/R$, where
		\[
		C^{1,+}_0(\Omega_{R}):=\left\{\varphi\in C^1(\Omega)\mid \varphi_{|_{\partial B_R\cap\mathbb{R}^N_+}}=0\right\}.
		\]
		Defining the vector field $\mathbf{v}_{_R}:\Omega_R\rightarrow\mathbb{R}^N$ by
		\begin{equation}\label{vector}
		    \mathbf{v}_{_R}(x):=\left[(x\cdot\nabla u)|\nabla u|^{p-2}\nabla u - \frac{1}{p}x|\nabla u|^p\right]\varphi_{_R}, \ \ x\in\Omega_R,
		\end{equation}
		which is continuous by virtue of our hypotheses. Since
		$$
		-\Delta_p u=f(u)\ \mbox{ in }\ \Omega_{R,0},
		$$
		from \eqref{boa},  we have
		\begin{multline}\label{boa2}
		    -\mathrm{div}(\mathbf{v}_R)=\\\left[f(u(x))(x\cdot\nabla u)+\frac{N-p}{p}|\nabla u|^p\right]\varphi_R+\left[(x\cdot\nabla u)|\nabla u|^{p-2}\nabla u - \frac{1}{p}x|\nabla u|^p\right]\nabla\varphi_R,
		\end{multline}
		almost everywhere in the open subset $\Omega_{R,0}$. Next, we are going to verify that \eqref{boa2} holds in $\Omega_R$. Note that for any $\psi\in C^{1}_0(\Omega_R)$ we have
		\begin{align*}
			\nonumber-\int_{\Omega_R}\mathbf{v}_{_R}\cdot \nabla \psi\,\mathrm{d}x &=\int_{\Omega_R}\left[\left(f(u)(x\cdot\nabla u)+\frac{N-p}{p}|\nabla u|^p\right)\varphi_{_R}\right.\\
			&+\left.\left((x\cdot\nabla u)|\nabla u|^{p-2}\nabla u - \frac{1}{p}x|\nabla u|^p\right)\nabla\varphi_{_R}\right]\psi\,\mathrm{d}x.
		\end{align*}
		Consequently, \eqref{boa2} holds in the distribution sense in $\Omega_R$. For this, applying a version of the divergence theorem (c.f. \cite[Lemma A.1]{CT}) to the vector field $\mathbf{v}_R$, we obtain
		\begin{align*}
			\int_{\Gamma_R}\mathbf{v}_{_R}\cdot\nu(x)\,\mathrm{d}x^\prime &=\int_{\Omega_R}\left[f(u)(x\cdot\nabla u)+\frac{N-p}{p}|\nabla u|^p\right]\varphi_{_R}\,\mathrm{d}x\\
			&+\int_{\widetilde{\Omega}_R}\left[(x\cdot\nabla u)|\nabla u|^{p-2}\nabla u - \frac{1}{p}x|\nabla u|^p\right]\nabla\varphi_{_R}\,\mathrm{d}x,
		\end{align*}
		where $\widetilde{\Omega}_R:=B_{R}\setminus B_{R/2}$. Since $|\nabla \varphi_{_R}(x)|\leq2/R$ and $|\nabla u|\in L^p(\mathbb{R}^N)$ we get   
		\begin{align*}
			\left|\int_{\widetilde{\Omega}_R}\left[(x\cdot\nabla u)|\nabla u|^{p-2}\nabla u - \frac{1}{p}x|\nabla u|^p\right]\nabla\varphi_{_R}\,\mathrm{d}x\right|&\leq\frac{p+1}{p}\int_{\widetilde{\Omega}_R}|x||\nabla u|^p|\nabla\varphi_R|\,\mathrm{d}x\\
			&\leq\frac{2(p+1)}{p}\int_{\widetilde{\Omega}_R}|\nabla u|^p\,\mathrm{d}x\rightarrow 0,
		\end{align*}
		as $R\rightarrow+\infty$. Let $(R_k)_{k\in\mathbb{N}}$  be a sequence such that $R_k\rightarrow+\infty$ as $k\rightarrow+\infty$ and denote by $\varphi_k:=\varphi_{_{R_k}}\in C^{1,+}_0(\Omega_{R_k})$. 
		For simplicity let us denote by $\mathbf{v}_{k}$ the vector field $\mathbf{v}_{R_k}$ defined in \eqref{vector}. Hence, the above identity turns 
		\begin{equation}\label{limite}
			\int_{\Gamma_{R_k}}g(u)(x\cdot\nabla u) \varphi_k\,\mathrm{d}x^\prime =\int_{\Omega_{R_k}}\left[f(u)(x\cdot\nabla u)+\frac{N-p}{p}|\nabla u|^p\right]\varphi_k\,\mathrm{d}x+o(R_k).
		\end{equation}
		Now, to pass the limit in \eqref{limite}, we observe that using integration by parts,
		\begin{align*}
			\int_{\Omega_{R_k}}f(u)(x\cdot\nabla u)\varphi_k\,\mathrm{d}x
			&=\sum_{i=1}^N\int_{\Omega_{R_k}}\left(F(u)\right)_{x_i}x_i\varphi_k\,\mathrm{d}x\\
			&=\int_{\Omega_{R_k}}\left[-NF(u)\varphi_k-F(u)(x\cdot\nabla\varphi_k)\right]\,\mathrm{d}x.
		\end{align*}
		Clearly, by the dominated convergence theorem one has	$$\lim_{k\rightarrow+\infty}\int_{\Omega_{R_k}}F(u)\varphi_k\,\mathrm{d}x=\int_{\mathbb{R}^N_+}F(u)\,\mathrm{d}x,
		$$
		and using the estimate on $\nabla\varphi_k$ and the hypothesis of integrability on $F(u)$, we get
		\begin{align*}
		\left|\int_{\Omega_{R_k}}F(u)(x\cdot\nabla\varphi_k)\,\mathrm{d}x\right|&=\left|\int_{\widetilde{\Omega}_{R_k}}F(u)(x\cdot\nabla\varphi_k)\,\mathrm{d}x\right|\\&\leq 2\int_{\widetilde{\Omega}_{R_k}}|F(u)|\,\mathrm{d}x\rightarrow0,\ \ \mbox{ as }\ \ k\rightarrow+\infty.
		\end{align*}
		
		Observing that 
		$$
		\begin{aligned}
		\int_{\Gamma_{R_k}}&g(u)(x\cdot\nabla u) \varphi_k\,\mathrm{d}x^\prime=\\
		&\sum_{i=1}^{N-1}\int_{\Gamma_{R_k}}\left(G(u)\right)_{x_i}x_i\varphi_k\,\mathrm{d}x^\prime
		&=\int_{\Gamma_{R_k}}\left[-(N-1)G(u)\varphi_k-G(u)\left(x^\prime\cdot
		\nabla\varphi_k\right)\right]\,\mathrm{d}x^\prime,
		\end{aligned}
		$$
		and using a similar argument, we have
		$$
		\lim_{k\rightarrow+\infty}\int_{\Gamma_{R_k}}G(u)\varphi_k\,\mathrm{d}x^\prime= \int_{\mathbb{R}^{N-1}}G(u)\,\mathrm{d}x^\prime\quad\mbox{and}\quad \lim \int_{\Gamma_{R_k}}G(u)\left(x^\prime\cdot
		\nabla\varphi_k\right)\,\mathrm{d}x^\prime=0.
		$$
		Furthermore, by the dominated convergence theorem,
		$$
		\lim_{k\rightarrow+\infty}\int_{\Omega_{R_k}}|\nabla u|^p\varphi_k\,\mathrm{d}x=\int_{\mathbb{R^N_+}}|\nabla u|^p\,\mathrm{d}x.
		$$
		Taking the limit in \eqref{limite} and using the above estimates, we obtain the desired result.
	\end{proof}
	\begin{remark}
		Optimal regularity results were obtained in \cite{ITa} for Neumann boundary conditions since the solutions are at least $C^1$ in a bounded domain. In our case, this hypothesis is trivially satisfied in view of Lemma \ref{13}. The assumption $|\nabla u|^{p-1}\in W^{1,1}_{\mathrm{loc}}(\mathbb{R}^N_+)$ was proved in \cite{Lou}. Furthermore, it is worth mentioning that in Proposition \ref{Pohozaev}, no hypothesis about decay is required.
	\end{remark}
	
	\begin{proof}[Proof of Theorem~\ref{T0}]Let $u_\lambda\in E\cap C_{\mathrm{loc}}^1(\overline{\mathbb{R}^N_+})$ be a weak solution of Problem \eqref{P1} with defined signal. Applying Proposition \ref{Pohozaev} with $f=0$ and $g(t)=|t|^{q-2}t-\lambda |t|^{p-2}t$ we obtain
		\[
		\lambda\left(\frac{p-1}{p}\right)\int_{\mathbb{R}^{N-1}}|u_\lambda|^p\mathrm{d}x'=\left( \frac{N-1}{q}+1-\frac{N}{p} \right)\int_{\mathbb{R}^{N-1}}|u_\lambda|^q\mathrm{d}x'. 
		\]
		Consequently, $i)$, $ii)$, and $iii)$ follow immediately from the above identity. To prove $iv)$, notice that, since $u$ is nontrivial solution to Problem \eqref{P1}, we have 
		\[
		\int_{\mathbb{R}^{N-1}}\left(|u_\lambda|^q-\lambda|u_\lambda|^p\right)\mathrm{d}x'=\int_{\mathbb{R}^{N}_+}|\nabla u_\lambda|^p\mathrm{d}x, 
		\]
		and hence, there exists $x_0\in\mathbb{R}^{N-1}$ such that 
		$
		|u_\lambda(x_0)|^q-\lambda|u_\lambda(x_0)|^p\geq0$ which implies that $\lambda\leq |u_\lambda(x_0)|^{q-p}\leq\|u_\lambda\|^{q-p}_{L^\infty(\mathbb{R}^{N-1})}$ and this concludes the proof.
	\end{proof}
	\begin{remark}
		When $f(u)=\lambda|u|^{p-2}u$ and $g(u)=|u|^{q-2}u$, in \cite[Theorem 1.7]{rossi} was proved a nonexistence result of solutions with some restrictive hypotheses, such as decaying and smoothness. Using Proposition \ref{Pohozaev}, we can remove all these assumptions. Observe that the function $u(x)=e^{-x_N}$ is a solution of \eqref{P1} for every $q$ with $\lambda=p-1$ that does not satisfies the integrability conditions in Proposition~\ref{Pohozaev}.
	\end{remark}

	\section{Symmetry of solutions}
In this section, we prove that weak positive solutions of Problem \eqref{P1} are radially symmetric concerning the first $N-1$ variables. The Moving planes method plays a fundamental role in establishing some qualitative properties of positive solutions of nonlinear elliptic equations like symmetry and monotonicity. Using these ideas, E. Abreu et al. \cite[Proposition 4.1]{EMEDOOEVE2010} proved that every weak positive solution of Problem \eqref{P1} is radially symmetric concerning the first $N-1$ variables when $p=2$. The fact that the Laplacian operator is linear and non-degenerate played a fundamental role in their proof. The proof of Theorem~\ref{simetria} relies on the following lemmas. 

  \medskip

	\begin{lemma}\label{negative}
		Let $a$ and $b$ real numbers satisfying $0<b<a<e^{-1/p}$. Then the function $f:[p,+\infty)\rightarrow\mathbb{R}$ defined by
		$f(t):=a^t-b^t$
		is decreasing.
	\end{lemma}
	\begin{proof}
		Note that the derivative of $f$ is given by $f'(t)=\ln(a)a^t-\ln(b)b^t=g(a)-g(b)$, where  
		$g(s):=s^t\ln(s)$ for 
		$s>0$.
		By a direct computation we see that $g'(s)\leq0$ for all $s\leq e^{-1/p}$ and the conclusion follows.
	\end{proof}
	Now, we need to introduce some basic notation and terminology. For each $\lambda >0$, we consider 
	\begin{align*}
		\mbox{$T$}_\lambda&=\left\{ x= (x_1,\overline{x},x_N)\in \mathbb{R}\times\mathbb{R}^{N-2}\times\mathbb{R}_{+}\text{ : } x_1=\lambda \right\} \\
		\mbox{${\sum}$}_\lambda &=\left\{ x= (x_1,\overline{x},x_N)\in \mathbb{R}\times\mathbb{R}^{N-2}\times\mathbb{R}_{+}\text{ : } x_1>\lambda  \right\},\\
		\mbox{${\sum}$}_{\lambda,0}&=\left\{x\in\mbox{${\sum}$}_\lambda\text{ : }x_N=0\right\} \\ x_\lambda &=(2\lambda -x_1,\overline{x},x_N), \ \ u_\lambda(x)=u(x_\lambda) \mbox{ and }\ e_\lambda=(2\lambda,\overline{0},0),\\
		U_\lambda(x)&=u(x)-u_\lambda(x)\ \mbox{and}\ \ U_{\lambda,+}=\max\{U_\lambda,0\}.
	\end{align*}
	Our first step is to prove the following lemma.

 \medskip
 
	\begin{lemma}
		For $\lambda$ large enough we have $U_\lambda(x)\leq0$ for all $x\in{\sum}_\lambda$.
	\end{lemma}
	\begin{proof}	
		We start rewriting \eqref{P1} as 
		\begin{align}\label{sym:2}
			\left\{
			\begin{alignedat}{2}
				-\Delta_p u+\Delta_pu_\lambda&=0&\quad&\text{ in }\quad \mathbb{R}^N_+,\\
				\vert \nabla u \vert^{p-2}\frac{\partial u}{\partial \nu}-\vert \nabla u_\lambda \vert^{p-2}\frac{\partial u_\lambda}{\partial \nu}&= (u^+)^{q-1}-u^{p-1}-(u^+)_\lambda^{q-1}+ u_\lambda^{p-1}&\quad&\text{ on }\quad\partial\mathbb{R}_{+}^{N}.
			\end{alignedat}
			\right.
		\end{align}
		Let $\epsilon$ be a positive real number small enough and consider 
		 $\phi\in C^1_c(\mathbb{R})$ a cut-off function such that $0\leq\phi(t)\leq1$,
		\begin{align}\label{funcoes}
			\phi(t)=\left\{
			\begin{array}{ll}
				0 & \mbox{if}\ \ |t|\leq 1\ \ \mbox{or}\ \  |t|\geq2\epsilon^{-2}\\
				1& \mbox{if}\ \ 2\leq|t|\leq\epsilon^{-2}.
			\end{array}
			\right.
			&\mbox{and}&
			\phi'(t)\leq\left\{
			\begin{array}{ll}
				2 & \mbox{if}\ \ 1\leq|t|\leq2,\\
				2\epsilon^2& \mbox{if}\ \ \epsilon^{-2}\leq|t|\leq2\epsilon^{-2}.
			\end{array}
			\right.
		\end{align}
		Now, setting $\phi_\epsilon(x):=\phi\left(|x-e_\lambda|\epsilon^{-1}\right)$ with $\epsilon<2/\lambda$, we have the function $\varphi_\epsilon=\phi_\epsilon\cdot U_{\lambda,+}$ satisfies $\mathrm{supp}(\varphi_\epsilon)\cap\overline{\mathbb{R}^N_+} \subset\overline{{\sum}_\lambda}$. By using it as  test function in \eqref{sym:2} we have
		\begin{multline*}
		\int_{\sum_\lambda}\phi_\epsilon\left( \vert \nabla u \vert^{p-2}\nabla u- \vert \nabla u_\lambda \vert^{p-2}\nabla u_\lambda\right)\nabla U_{\lambda,+}=\\ \int_{\sum_{\lambda,0}}\phi_\epsilon\left( u^{q-1}-u^{p-1}-u_\lambda^{q-1}+ u_\lambda^{p-1}\right)U_{\lambda,+}\\
			-\int_{\sum_\lambda}\left( \vert \nabla u \vert^{p-2}\nabla u- \vert \nabla u_\lambda \vert^{p-2}\nabla u_\lambda\right)\nabla\phi_\epsilon U_{\lambda,+}.
		\end{multline*}
		Setting $\sum_{\lambda,\epsilon}:=\sum_\lambda\cap (B_{\frac{1}{\epsilon}}(e_\lambda)\backslash B_{2\epsilon}(e_\lambda))$, by the definition of $\phi$ we get
		\begin{multline*}
			\int_{\sum_{\lambda,\epsilon}}\left( \vert \nabla u \vert^{p-2}\nabla u- \vert \nabla u_\lambda \vert^{p-2}\nabla u_\lambda\right)\nabla U_{\lambda,+}\leq\\ \int_{\sum_{\lambda,0}}\left( u^{q-1}-u_\lambda^{q-1}-(u^{p-1}- u_\lambda^{p-1})\right)U_{\lambda,+}\\
			+\int_{\sum_\lambda}\left|\vert \nabla u \vert^{p-2}\nabla u- \vert \nabla u_\lambda \vert^{p-2}\nabla u_\lambda\right|\vert\nabla\phi_\epsilon\vert U_{\lambda,+}.
		\end{multline*}
		For $\lambda$ large enough, from the estimate given in Theorem \ref{T-Regularity-Decay} and using that $p<q$ together with the monotonicity in Lemma~\ref{negative} we have
		\begin{multline}\label{sym:2.1}
			\int_{\sum_{\lambda,\epsilon}}\left( \vert \nabla u \vert^{p-2}\nabla u- \vert \nabla u_\lambda \vert^{p-2}\nabla u_\lambda\right)\nabla U_{\lambda,+}\leq\\ \int_{\sum_\lambda}\left|\vert \nabla u \vert^{p-2}\nabla u- \vert \nabla u_\lambda \vert^{p-2}\nabla u_\lambda\right|\vert\nabla\phi_\epsilon\vert U_{\lambda,+}.
		\end{multline}
		Applying H\"{o}lder's inequality in the above inequality, we get
		\begin{align}\label{sym:3}
			\int_{\sum_{\lambda,\epsilon}}\left( \vert \nabla u \vert^{p-2}\nabla u-\right.& \left.\vert \nabla u_\lambda \vert^{p-2}\nabla u_\lambda\right)\nabla U_{\lambda,+}\nonumber\\
			&\leq \left(\vert\vert \nabla u \vert\vert_{L^p(\sum_\lambda)}^{p-1}+ \vert\vert \nabla u_\lambda \vert\vert_{L^p(\sum_\lambda)}^{p-1}\right)\vert\vert U_{\lambda,+}\nabla\phi_\epsilon \vert\vert_{L^p(\sum_\lambda)} .
		\end{align}
		Defining $\tilde{B_\epsilon}:=\{x\in\sum_\lambda \text{ : } \epsilon<\vert x-e_\lambda\vert<2\epsilon, \ \mbox{or}\ \epsilon^{-1}<\vert x-e_\lambda\vert<(2\epsilon)^{-1}\}$, from (\ref{funcoes}) and (\ref{sym:3}) we have again by Theorem \ref{T-Regularity-Decay} and H\"{o}lder's inequality
		\begin{multline}\label{sym:4}
			\vert\vert\nabla\phi_\epsilon\cdot  U_{\lambda,+}\vert\vert_{L^p(\sum_\lambda)}\leq\\ \vert\vert\nabla\phi_\epsilon  \vert\vert_{L^N(\tilde{B_\epsilon})}   \vert\vert U_{\lambda,+}  \vert\vert_{L^{p^*}(\tilde{B_\epsilon})} \leq \\C\left(\int_{\tilde{B_\epsilon}}\left(u-u_\lambda\right)^{p^*}_+\right)^{(N-p)/(pN)}\rightarrow 0,
		\end{multline}
		as $\epsilon\rightarrow0$. If $p\geq2$, letting $\epsilon\rightarrow 0$ in \eqref{sym:2.1}, \eqref{sym:3} and \eqref{sym:4}, we conclude that exists a constant $c_p>0$ such that
		\begin{align*}
			c_p \vert \vert\nabla( u- u_\lambda)\vert \vert^p_{L^p(\sum_\lambda)}\leq \int_{\sum_{\lambda}}\left( \vert \nabla u \vert^{p-2}\nabla u- \vert \nabla u_\lambda \vert^{p-2}\nabla u_\lambda\right)\nabla U_{\lambda,+} \leq 0,
		\end{align*} 
        where in the last inequalities, we used a classical inequality (see, for instance, \cite[Lemma 2.1]{Simon}). If $1<p<2$ we have
		\begin{multline*}
			c_p \dfrac{\vert \vert\nabla( u- u_\lambda)\vert \vert^2_{L^2(\sum_\lambda)}}{(\vert \vert\nabla u\vert\vert_{L^2(\sum_\lambda)} +\vert\vert \nabla u_\lambda\vert \vert_{L^2(\sum_\lambda)})^{2-p}}\\
			\leq \int_{\sum_{\lambda}}\left( \vert \nabla u \vert^{p-2}\nabla u- \vert \nabla u_\lambda \vert^{p-2}\nabla u_\lambda\right)\nabla U_{\lambda,+} \leq 0,
		\end{multline*} 
		where in the last inequality, we used Simon's inequality again. Therefore, we obtain $\vert \nabla( u- u_\lambda)\vert=0$ in $\mbox{\normalsize${\sum}$}_\lambda$ and consequently $u- u_\lambda=c$ in $\sum_\lambda$. Thus, $c\leq c+u_\lambda=u\rightarrow0$ as $|x|\rightarrow\infty$ which implies the result.
	\end{proof}
	Let us define
	\begin{align*}
		\Lambda:=
		\inf\left\{\lambda>0 :  U_\mu\leq0, \ \forall x\in\mbox{\normalsize${\sum}$}_\mu,\ \mu>\lambda\right
		\},
	\end{align*}
	we have the following result.

\medskip
 
	\begin{lemma}\label{sym:5}
		The number $\Lambda=0$.
	\end{lemma}
	\begin{proof}
		Assume by contradiction that $\Lambda>0$. By continuity, we have $U_\Lambda\leq0$ for all $x\in\mbox{\normalsize${\sum}$}_\Lambda$ and consequently
		\begin{equation}\label{Simetria1}
			\left\{
			\begin{alignedat}{2}
				-\Delta_p u+ \Delta_pu_\Lambda=0& &\quad&\text{ in }\quad \mbox{\normalsize${\sum}$}_\Lambda,\\
				u -u_\Lambda=0& &\quad&\text{ in }\quad \mbox{\Large $T$}_\Lambda,\\
				\vert \nabla u \vert^{p-2}\frac{\partial u}{\partial \nu}+u^{p-1}= u^{q-1}&\leq u_\Lambda^{q-1}= u_\Lambda^{p-1}+\vert \nabla u_\Lambda \vert^{p-2}\frac{\partial u_\Lambda}{\partial \nu}&\quad&
				\text{ on }\quad\mbox{\normalsize${\sum}$}_{\Lambda,0}.
			\end{alignedat}
			\right.
		\end{equation}
		We will begin an analysis of the possibilities.
		
		\medskip
		
		\noindent \textit{Case 1:} $u_{x_1}(x)<0$ uniformly for all $x\in\mbox{$T$}_\Lambda$. In this case exists a $\delta>0$ such that the function $g(x_1)=u(x_1,\overline{x},t)$ is decreasing  in  
		\[\mbox{$T$}^{\delta}_\Lambda=\left\{x\in\mbox{$T$}_\Lambda\text{ : } \Lambda-\delta<x_1<\Lambda+\delta\right\},\]
		which contradicts the minimality of $\Lambda$.
		
		\medskip 
		
		\noindent \textit{Case 2:} $u_{x_1}(x)<0$, but not uniformly for all $x\in\mbox{$T$}_\Lambda$.	This mean that exists a sequence $(x^k)\subset\mbox{\normalsize${\sum}$}_\Lambda$, with $x^k=(x_{1,k},\overline{x}_k,x_{N_k})$, such that $|\overline{x}_k|^2+x_{N_k}^2\rightarrow+\infty$, $u_{x_1}(x^k)=0$,
		and $|\Lambda-x_{1,k}|\rightarrow0$ as $k$ goes to infinity. Besides, we have for $\epsilon>0$ small enough that $U_{\Lambda-\epsilon}(x^k)=0$. Indeed, if $U_{\Lambda-\epsilon}(x^k)>0$, then
		exists a $\delta>0$ such that $U_{\Lambda-\epsilon}(x)>0$ for all $x\in B(x^k_{\Lambda-\epsilon},\delta)$. Thus, the function $g(x_1) =u(x_1,\overline{x},x_N)$ is increasing on the left hand side of 
		$x_{1,k}$, which is a contradiction with $u_{x_1}(x_1,\overline{x}_k,x_{N_k})<0$, for $x_1<x_{1,k}$. If we assume that $U_{\Lambda-\epsilon}(x^k)<0$, then $u_{x_1}(x^k)=0$ implies that $x_{1,k}$ is an inflection point, and again we can follow as in the previous case. From this assumptions we have that exists a neighbourhood $B_k$, such that $x^k\in\partial B_k$, and the unit outward normal in $x^k$ given by $\nu=(-1,\overline{0},0)$. Also we can ask that $U_{\Lambda-\epsilon}(x^k)>0$ in $B_k$, and $u(x)>u(x^k_{\Lambda-\epsilon})$ for $x\in B_k$. So, the function $v(x)=u(x)-u(x^k_{\Lambda-\epsilon})$ is positive in $B_k$ and $v(x^k)=0$. Besides, $\Delta_p v=0$ in $B_k$. By the strong maximum principle (see \cite{vazquez}), the normal derivative has a negative sign,
		\[0>\frac{\partial v}{\partial \nu}(x^k)=-2u_{x_1}(x^k)=0,\]
		which is a contradiction.
		
		\noindent \textit{Case 3:} $u_{x_1}(x)\geq 0$ for all $x\in\mbox{$T$}_\Lambda$. From Picone's identity \cite{alegretto}, we have
		\begin{align*}
			0\leq&\int_{{\sum}_\Lambda}\left(\vert\nabla u\vert^p+(p-1)\frac{u^p}{u_\Lambda^p}\vert \nabla u_\Lambda\vert^p-p\frac{u^{p-1}}{u_\Lambda^{p-1}}\nabla u\vert 
			\nabla u_\Lambda\vert^{p-2}\nabla u_\Lambda\right)\\
			=&\int_{{\sum}_\Lambda}\left(\vert\nabla u\vert^p-\nabla\left(\frac{u^p}{u_\Lambda^{p-1}}\right)\vert \nabla u_\Lambda\vert^{p-2}\nabla u_\Lambda\right)=A.
		\end{align*}
		From \eqref{Simetria1} and using integration by parts we have 
		\begin{align*}
			A=&\int_{{\sum}_{\Lambda,0}}\left(u^q-u^p-u^pu_\Lambda^{q-p}+u^p\right)+\int_{\mbox{$T$}_\Lambda}\left(u\vert\nabla u\vert^{p-2}\frac{\partial u}{\partial \nu}-u_\Lambda\vert \nabla u_\Lambda
			\vert^{p-2}\frac{\partial u_\Lambda}{\partial \nu}\right)\\
			=&\int_{{\sum}_{\Lambda,0}}u^p\left(u^{q-p}-u_\Lambda^{q-p}\right)+\int_{\mbox{$T$}_\Lambda}u\vert\nabla u\vert^{p-2}\left(-2u_{x_1}\right)\leq0.
		\end{align*}
		As a consequence, there exists a constant $\theta$ such that $u=\theta u_\Lambda$ in ${\sum}_\Lambda$. Now, since $U_\Lambda\leq0$ we
		conclude that $\theta\leq1$. On the other hand,
		\[u_{x_1}=-\theta u_{x_1},\ u_{x_2}=\theta u_{x_2},\ \ldots,\ u_{x_N}=\theta u_{x_N},\ \ \mbox{on}\ \ \mbox{${T}$}_\Lambda.\]
		Since the function $u$ is not a constant, we conclude that $\theta \neq0$. Thus, $u_{x_1}\equiv0$. Now, assume that $u_{x_j}=0$, for $j=2,\ldots,N$ on $\mbox{$T$}
		_\Lambda$. This implies that $u$ is constant on 
		$\mbox{$T$}_\Lambda$, which is not possible. Therefore, $u_{x_j}\neq0$, for some $j=2,\ldots,N$, on $\mbox{$T$}_\Lambda$, 
		and we get that $\theta=1$. This gives the conclusion that $U_\Lambda\equiv0$. The argument above shows that $\Lambda=0$. Indeed, since $u_{x_1}=0$ on ${\sum}_\Lambda$, we conclude that $u$ is constant in $x_1$. But this implies that 
		$U_{\Lambda-\epsilon}\equiv0$, for $\epsilon>0$.
		By Theorem 
		\ref{T-Regularity-Decay}, for $(x_1,\overline{x}_0,x_{N_0})\in\mathbb{R}\times\mathbb{R}^{N-2}\times\mathbb{R}_+\cap{\sum}_\Lambda$ we have
		\begin{multline*}
		0<c(x_0,x_{N_0})=u(x_1,\overline{x}_0,x_{N_0})=\\O\left((x^2_1+|\overline{x}_0|^2+x_{N_0}^2)^{(p-N)/(2(p-1))}\right)\longrightarrow 0 \ \mbox{as}\ \ x_1\rightarrow+\infty,
  \end{multline*}
		which is a contradiction. This completes the proof of Lemma~\ref{sym:5}.
	\end{proof}
	
	\begin{proof}[Finalizing the proof of Theorem \ref{simetria}]
		According to Lemma \ref{sym:5} we have $\Lambda\equiv0,$ which means by the continuity,
		\[
		u(x^\prime,x_N)\leq u_0(x^\prime,x_N)\ \  \mbox{for all}\ \ (x^\prime,x_N)\in \Sigma_0.
		\]
		Thus, applying the Moving plane method to the function $V_{\lambda}=u_\lambda-u$, 
		\[
		u_0(x^\prime,x_N)\leq u(x^\prime,x_N)\ \  \mbox{for all}\ \ (x^\prime,x_N)\in \Sigma_0.
		\]
		Therefore, $u$ is symmetric with respect to the $x_1$ direction. Applying the same procedure to the variables $x_2,\dots, x_{N-1}$, we get that $u$ is radially symmetric around the point $(0',0)\in\mathbb{R}^{N}_+$ and this finishes the proof.
	\end{proof}

 \begin{flushleft}
 {\bf Funding:}  J.M.~do~\'O acknowledges partial support  from 
	CNPq through grants 312340/2021-4  and 429285/2016-7 and Para\'iba State Research Foundation (FAPESQ), grant no 3034/2021.  E. Medeiros acknowledges partial support  from CNPq through grant 308900/2019-7. R. Clemente acknowledges partial support from CNPq through grant 304454/2022-2.\\ 
\end{flushleft}


\end{document}